\documentclass[reqno,a4paper, 12pt]{amsart}

\usepackage{pictexwd,dcpic,amssymb}
\usepackage{graphicx}
\usepackage{array}
\usepackage{multirow}

\newtheorem{theorem}{Theorem}[section]
\newtheorem{proposition}[theorem]{Proposition}
\newtheorem{lemma}[theorem]{Lemma}

\theoremstyle{definition}

\newtheorem{example}[theorem]{Example}

\theoremstyle{remark}
\newtheorem{remark}[theorem]{Remark}

\numberwithin{equation}{section}

\begin{document}

\title{Ring class invariants over imaginary quadratic fields}

\author{Ick Sun Eum}
\address{Department of Mathematical Sciences, KAIST}
\curraddr{Daejeon 373-1, Korea} \email{zandc@kaist.ac.kr}
\thanks{}

\author{Ja Kyung Koo}
\address{Department of Mathematical Sciences, KAIST}
\curraddr{Daejeon 373-1, Korea} \email{jkkoo@math.kaist.ac.kr}
\thanks{}

\author{Dong Hwa Shin}
\address{Department of Mathematical Sciences, KAIST}
\curraddr{Daejeon 373-1, Korea} \email{shakur01@kaist.ac.kr}
\thanks{}


\keywords{Class field theory, complex multiplication, elliptic and
modular units, modular functions.
\newline This research was partially supported by Basic Science Research
Program through the NRF of Korea funded by MEST (2010-0001654). The
third named author is partially supported by TJ Park Postdoctoral
Fellowship.}

\begin{abstract}
We show by adopting Schertz's argument with the Siegel-Ramachandra
invariants that the singular values of certain $\Delta$-quotients
generate ring class fields over imaginary quadratic fields.
\end{abstract}

\maketitle

\section{Introduction}

In number theory ring class fields over imaginary quadratic fields,
more exactly, primitive generators of ring class fields as real
algebraic integers play an important role in the study of certain
quadratic Diophantine equations. For example, let $n$ be a positive
integer and $H_\mathcal{O}$ be the ring class field of the order
$\mathcal{O}=\mathbb{Z}[\sqrt{-n}]$ in the imaginary quadratic field
$K=\mathbb{Q}(\sqrt{-n})$. If $p$ is an odd prime not dividing $n$,
then we have the following assertion:
\begin{eqnarray*}
&&\textrm{$p=x^2+ny^2$ is solvable for some integers $x$ and $y$}\\
&\Longleftrightarrow&\textrm{the Legendre symbol $(-n/p)=1$ and
$f_n(X)\equiv0\pmod{p}$
has an integer solution,}\\
&&\textrm{where $f_n(X)$ is the minimal polynomial of a real
algebraic integer $\alpha$}\\
&&\textrm{which generates $H_\mathcal{O}$ over $K$ (\cite[Theorem
9.2]{Cox}).}
\end{eqnarray*}
\par
Given an imaginary quadratic field $K$ with the ring of integer
$\mathcal{O}_K=\mathbb{Z}[\theta]$ such that $\theta\in\mathfrak{H}$
(= the complex upper half-plane), let $\mathcal{O}=[N\theta,1]$ be
the order of conductor $N$ ($\geq1$) in $K$. We know a classical
result from the theory of complex multiplication that the
$j$-invariant $j(\mathcal{O})=j(N\theta)$ generates the ring class
field $H_\mathcal{O}$ over $K$ (\cite[Theorem 11.1]{Cox} or
\cite[Chapter 10 Theorem 5]{Lang}). We have an algorithm to find the
minimal polynomial (= class polynomial) of such a generator
$j(\mathcal{O})$ (\cite[$\S$13.A]{Cox}), however, its coefficients
are too gigantic to handle for practical use.
\par
Unlike the classical case, Chen-Yui (\cite{C-Y}) constructed a
generator of the ring class field of certain conductor in terms of
the singular value of the Thompson series which is a Hauptmodul for
$\Gamma_0(N)$ or $\Gamma_0^\dagger(N)$, where
$\Gamma_0(N)=\{\gamma\in\mathrm{SL}_2(\mathbb{Z}):\gamma\equiv\left(\begin{smallmatrix}
* & * \\0&*\end{smallmatrix}\right)\pmod{N}\}$ and
$\Gamma_0^\dagger(N)=\langle\Gamma_0(N),\left(\begin{smallmatrix}0 &
-1/\sqrt{N}\\\sqrt{N} & 0\end{smallmatrix}\right)\rangle$ in
$\mathrm{SL}_2(\mathbb{R})$. In like manner, Cox-Mckay-Stevenhagen
(\cite{C-M-S}) showed that certain singular value of a Hauptmodul
for $\Gamma_0(N)$ or $\Gamma_0^\dagger(N)$ with rational Fourier
coefficients generates $H_\mathcal{O}$ over $K$. And, Cho-Koo
(\cite[Corollaries 4.4 and 4.5]{C-K}) recently revisited and
extended these results by using the theory of Shimura's canonical
models and his reciprocity law.
\par
On the other hand, Ramachandra (\cite[Theorem 10]{Ramachandra})
showed that arbitrary finite abelian extension of an imaginary
quadratic field $K$ can be generated over $K$ by a theoretically
beautiful elliptic unit, but his invariant involves overly
complicated product of high powers of singular values of the Klein
forms and singular values of the $\Delta$-function to use in
practice. This motivates our work of finding simpler ring class
invariants in terms of the Siegel-Ramachandra invariant as Lang
pointed out in his book (\cite[p.292]{Lang}) in case of ray class
fields.
\par
More precisely, for any pair $(r_1,r_2)\in\mathbb{Q}^2-\mathbb{Z}^2$
we define the \textit{Siegel function} $g_{(r_1,r_2)}(\tau)$ on
$\mathfrak{H}$ by the following infinite product
\begin{eqnarray}\label{FourierSiegel}
g_{(r_1,r_2)}(\tau)=-q^{(1/2)\mathbf{B}_2(r_1)}e^{\pi
ir_2(r_1-1)}(1-q_z) \prod_{n=1}^{\infty}(1-q^nq_z)(1-q^nq_z^{-1}),
\end{eqnarray}
where $\mathbf{B}_2(X)=X^2-X+1/6$ is the second Bernoulli
polynomial, $q=e^{2\pi i\tau}$ and $q_z=e^{2\pi iz}$ with
$z=r_1\tau+r_2$. As the singular values of Siegel functions we shall
define the Siegel-Ramachandra invariants in $\S$\ref{section2}. And,
by adopting Schertz's idea (\cite[Proof of Theorem 3]{Schertz}) we
shall determine certain class fields over $K$ generated by norms of
the Siegel-Ramachandra invariants (Theorem \ref{generator}). In the
case of ring class fields we are enable to express the norms as the
singular values of certain $\Delta$-quotients (Theorem \ref{main}),
where
\begin{equation}\label{Delta}
\Delta(\tau)=(2\pi i)^{12}q\prod_{n=1}^\infty(1-q^n)^{24}.
\end{equation}
For example, let $N=\prod_{k=1}^n p_k^{e_k}$ be a product of odd
primes $p_k$ which are inert or ramified in $K/\mathbb{Q}$. We
assume that
\begin{equation*}
e_k+1>2/r_k~(k=1,\cdots,n)\quad\textrm{and}\quad
\left\{\begin{array}{ll} \gcd(p_1,\omega_K)=1
& \textrm{if $n=1$}\vspace{0.1cm}\\
\gcd(\prod_{k=1}^np_k,\prod_{k=1}^n(p_k^{2/r_k}-1))=1 & \textrm{if
$n\geq2$,}
\end{array}\right.
\end{equation*}
where $r_k$ is the ramification index of $p_k$ in $K/\mathbb{Q}$ and
$\omega_K$ is the number of roots of unity in $K$. Then, certain
quotient of the singular values $\Delta((N/N_S)\theta)$, where $N_S$
are the products of $p_k$'s, becomes a generator of the ring class
field of the order $[N\theta,1]$ over $K$ (Remark \ref{mainremark}).
This result is a continuation of our previous work with $n=1$
(\cite[$\S$5]{J-K-S}).
\par
Note that Theorems \ref{generator} and \ref{main} heavily depend on
Lemma \ref{character} which requires the assumption
(\ref{hypothesis}). However, in $\S$\ref{section5} we shall develop
a lemma which substitutes for Lemma \ref{character} in order to
release from the assumption (\ref{hypothesis}) to some extent (Lemma
\ref{character3} and Remark \ref{newremark}). For example, let $K$
be an imaginary quadratic field other than $\mathbb{Q}(\sqrt{-1})$,
$\mathbb{Q}(\sqrt{-3})$, and $N$ ($\geq2$) be an integer with prime
factorization $N=\prod_{a=1}^A s_a^{u_a}\prod_{b=1}^B q_b^{v_b}
\prod_{c=1}^C r_c^{w_c}$, where each $s_a$ (respectively, $q_b$ and
$r_c$) splits (respectively, is inert and ramified) in
$K/\mathbb{Q}$. If
\begin{equation*}
4\sum_{a=1}^{A}\frac{1}{(s_a-1)s_a^{u_a-1}}
+2\sum_{b=1}^B\frac{1}{(q_b+1)q_b^{v_b-1}}
+2\sum_{c=1}^C\frac{1}{r_c^{w_c}}<1,
\end{equation*}
then one can also apply Theorem \ref{main} without assuming
(\ref{hypothesis}) (Theorem \ref{main2} and Remark \ref{newremark}).
\par
Lastly, by making use of our simple invariant developed in Theorem
\ref{main} we shall present three examples (Examples \ref{exer1},
\ref{exer2} and \ref{exer3}).

\section{Primitive generators of class fields}\label{section2}

In this section we shall investigate some class fields over
imaginary quadratic fields generated by norms of the
Siegel-Ramachandra invariants.
\par
For a given imaginary quadratic field $K$ we let
\begin{eqnarray*}
d_K&:&\textrm{the discriminant of $K$,}\\
\mathfrak{d}_K&:&\textrm{the different of $K/\mathbb{Q}$,}\\
\mathcal{O}_K&:&\textrm{the ring of integers of $K$,}\\
\omega_K&:&\textrm{the number of root of unity in $K$,}\\
I_K&:&\textrm{the group of fractional ideals of $K$,}\\
P_K&:&\textrm{the subgroup of $I_K$ consisting of principal ideals
of $K$.}
\end{eqnarray*}
And, for a nonzero integral ideal $\mathfrak{f}$ of $K$ we set
\begin{eqnarray*}
I_K(\mathfrak{f})&:&\textrm{the subgroup of $I_K$ consisting of
ideals relatively prime to $\mathfrak{f}$,}\\
P_{K,1}(\mathfrak{f})&:&\textrm{the subgroup of
$I_K(\mathfrak{f})\cap P_K$ generated by the principal ideals
$\alpha\mathcal{O}_K$}\\
&&\textrm{for which $\alpha\in\mathcal{O}_K$ satisfies $\alpha\equiv1\pmod{\mathfrak{f}}$,}\\
\mathrm{Cl}(\mathfrak{f})&:&\textrm{the
ray class group $(\mathrm{modulo}~{\mathfrak{f}}$)},~\textrm{namely,}~I_K(\mathfrak{f})/P_{K,1}(\mathfrak{f}),\\
C_0&:&\textrm{the unit class of $\mathrm{Cl}(\mathfrak{f})$,}\\
\omega(\mathfrak{f})&:&\textrm{the number of roots of unity in $K$
which
are $\equiv1\hspace{-0.2cm}\pmod{\mathfrak{f}}$,}\\
N(\mathfrak{f})&:&\textrm{the smallest positive integer in
$\mathfrak{f}$}.
\end{eqnarray*}
\par
By the \textit{ray class field $K_\mathfrak{f}$ modulo
$\mathfrak{f}$ of $K$} we mean a finite abelan extension of $K$
whose Galois group is isomorphic to $\mathrm{Cl}(\mathfrak{f})$ via
the Artin map $\sigma$, namely
\begin{equation*}
\sigma=\bigg(\frac{K_\mathfrak{f}/K}{\cdot}\bigg)~:~\mathrm{Cl}(\mathfrak{f})
\stackrel{\sim}{\longrightarrow}\mathrm{Gal}(K_\mathfrak{f}/K).
\end{equation*}
In particular, if $\mathfrak{f}=\mathcal{O}_K$, then we simply
denote $K_\mathfrak{f}$ by $H$ and call it the \textit{Hilbert class
field of $K$}. We have a short exact sequence
\begin{equation}\label{exact}
1\longrightarrow
\pi_\mathfrak{f}(\mathcal{O}_K)^*/\pi_\mathfrak{f}(\mathcal{O}_K^*)
\stackrel{\Phi_\mathfrak{f}}{\longrightarrow}\mathrm{Cl}(\mathfrak{f})\longrightarrow\mathrm{Cl}(\mathcal{O}_K)\longrightarrow1,
\end{equation}
where
\begin{equation*}
\pi_\mathfrak{f}~:~\mathcal{O}_K\longrightarrow\mathcal{O}_K/\mathfrak{f}
\end{equation*}
is the natural surjection and $\Phi_\mathfrak{f}$ is induced by the
homomorphism
\begin{eqnarray*}
\widetilde{\Phi}_\mathfrak{f}~:~\pi_\mathfrak{f}(\mathcal{O}_K)^*&\longrightarrow&\mathrm{Cl}(\mathfrak{f})\\
\pi_\mathfrak{f}(x)&\mapsto&[x\mathcal{O}_K],~\textrm{the class
containing $x\mathcal{O}_K$},
\end{eqnarray*}
whose kernel is $\pi_\mathfrak{f}(\mathcal{O}_K^*)$
(\cite[Proposition 3.2.3]{Cohen}).
\par
Let $\chi$ be a character of $\mathrm{Cl}(\mathfrak{f})$. We denote
by $\mathfrak{f}_\chi$ the conductor of $\chi$, namely
\begin{equation*}
\mathfrak{f}_\chi=\gcd(\mathfrak{g}~:~\chi=\psi\circ
(\mathrm{Cl}(\mathfrak{f})\rightarrow\mathrm{Cl}(\mathfrak{g}))~\textrm{for
some character}~\psi~\textrm~\textrm{of}~\mathrm{Cl}(\mathfrak{g})),
\end{equation*}
and let $\chi_0$ be the proper character of
$\mathrm{Cl}(\mathfrak{f}_\chi)$ corresponding to $\chi$. Similarly,
if $\chi'$ is any character of $\pi_\mathfrak{f}(\mathcal{O}_K)^*$,
then the conductor $\mathfrak{f}_{\chi'}$ of $\chi'$ is defined by
\begin{equation*}
\mathfrak{f}_{\chi'}=\gcd(\mathfrak{g}~:~\chi'=\psi'\circ(\pi_\mathfrak{f}(\mathcal{O}_K)^*\rightarrow\pi_\mathfrak{g}(\mathcal{O}_K)^*)~\textrm{for
some
character}~\psi'~\textrm~\textrm{of}~\pi_\mathfrak{g}(\mathcal{O}_K)^*).
\end{equation*}
\par
For a character $\chi$ of $\mathrm{Cl}(\mathfrak{f})$ we denote by
$\widetilde{\chi}$ the character of
$\pi_\mathfrak{f}(\mathcal{O}_K)^*$ defined by
\begin{equation*}
\widetilde{\chi}=\chi\circ\widetilde{\Phi}_\mathfrak{f}.
\end{equation*}
\par
If $\mathfrak{f}=\prod_{k=1}^n \mathfrak{p}_k^{e_k}$, then from the
Chinese remainder theorem we have an isomorphism
\begin{eqnarray*}
\iota~:~\prod_{k=1}^n
\pi_{\mathfrak{p}_k^{e_k}}(\mathcal{O}_K)^*\stackrel{\sim}{\longrightarrow}\pi_\mathfrak{f}(\mathcal{O}_K)^*,
\end{eqnarray*}
and natural injections and surjections
\begin{equation*}
\iota_k~:~\pi_{\mathfrak{p}_k^{e_k}}(\mathcal{O}_K)^*\hookrightarrow
\prod_{\ell=1}^n
\pi_{\mathfrak{p}_\ell^{e_\ell}}(\mathcal{O}_K)^*\quad\textrm{and}\quad
v_k~:~\prod_{\ell=1}^n
\pi_{\mathfrak{p}_\ell^{e_\ell}}(\mathcal{O}_K)^*\rightarrow
\pi_{\mathfrak{p}_k^{e_k}}(\mathcal{O}_K)^*\quad(k=1,\cdots,n),
\end{equation*}
respectively. Furthermore, we consider the characters
$\widetilde{\chi}_k$ of
$\pi_{\mathfrak{p}_k^{e_k}}(\mathcal{O}_K)^*$ defined by
\begin{equation*}
\widetilde{\chi}_k=\widetilde{\chi}\circ\iota\circ\iota_k\quad(k=1,\cdots,n).
\end{equation*}

\begin{lemma}\label{conductor}
The notation being as above, we have
\begin{itemize}
\item[(i)] $\mathfrak{f}_{\widetilde{\chi}}=\mathfrak{f}_\chi$.
\item[(ii)]
$\widetilde{\chi}\circ\iota=\prod_{k=1}^n\widetilde{\chi}_k\circ
v_k$.
\item[(iii)] If $\widetilde{\chi}_k\neq1$, then
$\mathfrak{p}_k|\mathfrak{f}_{\widetilde{\chi}}$.
\end{itemize}
\end{lemma}
\begin{proof}
(i) and (ii) are immediate by the definitions of conductors and
$\widetilde{\chi}$, $\widetilde{\chi}_k$, $\iota$, $v_k$.\\
(iii) Without loss of generality we may assume
$\widetilde{\chi}_n\neq1$. Suppose on the contrary
$\mathfrak{p}_n\nmid\mathfrak{f}_{\widetilde{\chi}}$. Then, by the
definition of $\mathfrak{f}_{\widetilde{\chi}}$ there is a character
$\psi'$ of $\mathrm{Cl}(\mathfrak{f}_{\widetilde{\chi}})$ which
makes the following diagram commutative:
\begin{equation*}
\begindc{\commdiag}[50]

\obj(7,3)[E]{$\prod_{k=1}^{n-1}\pi_{\mathfrak{p}_k^{e_k}}(\mathcal{O}_K)^*$}

\obj(3,3)[D]{$\prod_{k=1}^n\pi_{\mathfrak{p}_k^{e_k}}(\mathcal{O}_K)^*$}

\obj(6,2)[C]{$\pi_{\mathfrak{f}_{\widetilde{\chi}}}(\mathcal{O}_K)^*$}

\obj(4,2)[B]{$\pi_\mathfrak{f}(\mathcal{O}_K)^*$}

\obj(5,1)[A]{$\mathbb{C}^*$}

\obj(9,3)[F]{\hspace{-1.5cm}$\cong\pi_{\mathfrak{f}\mathfrak{p}^{-e_n}}(\mathcal{O}_K)^*$}

\mor{B}{A}{$\widetilde{\chi}$}[\atright,\solidarrow]

\mor{C}{A}{$\psi'$}

\mor{D}{B}{$\iota$}[\atright,\solidarrow]

\mor{E}{C}{$B$}

\mor{D}{E}{$A$}

\mor{B}{C}{$C$}
\enddc
\end{equation*}
where $A$, $B$ and $C$ are natural surjections. If $\sigma_n$ is an
element of $\pi_{\mathfrak{p}_n^{e_n}}(\mathcal{O}_K)^*$ such that
$\widetilde{\chi}_n(\sigma_n)\neq1$, then
\begin{eqnarray*}
1&\neq&\widetilde{\chi}_n(\sigma_n)=\widetilde{\chi}\circ\iota\circ\iota_n(\sigma_n)
=\widetilde{\chi}\circ\iota(1,\cdots,1,\sigma_n)\\
&=&(\psi'\circ C)\circ\iota(1,\cdots,1,\sigma_n)=\psi'\circ B\circ
A(1,\cdots,1,\sigma_n)=\psi'\circ B(1,\cdots,1)=1,
\end{eqnarray*}
which renders a contradiction. Therefore,
$\mathfrak{p}_n|\mathfrak{f}_{\widetilde{\chi}}$.
\end{proof}

If $\mathfrak{f}\neq\mathcal{O}_K$ and
$C\in\mathrm{Cl}(\mathfrak{f})$, we take any integral ideal
$\mathfrak{c}$ in $C$. Let $\mathfrak{f}\mathfrak{c}^{-1}=[z_1,z_2]$
with $z=z_1/z_2\in\mathfrak{H}$. We define the
\textit{Siegel-Ramachandra invariant} (of conductor $\mathfrak{f}$
at $C$) by
\begin{equation*}
g_\mathfrak{f}(C)=
g_{(a/N(\mathfrak{f}),b/N(\mathfrak{f}))}(z)^{12N(\mathfrak{f})},
\end{equation*}
where $a$, $b$ are integers such that
$1=(a/N(\mathfrak{f}))z_1+(b/N(\mathfrak{f}))z_2$. This value
depends only on the class $C$ and belongs to the ray class field
$K_\mathfrak{f}$ (\cite[Chapter 2 Proposition 1.3 and Chapter 11
Theorem 1.1]{K-L}). And, there is a well-known transformation
formula
\begin{equation}\label{Artin}
g_\mathfrak{f}(C_1)^{\sigma(C_2)}=g_\mathfrak{f}(C_1C_2)
\quad(C_1,C_2\in\mathrm{Cl}(\mathfrak{f}))
\end{equation}
(\cite[p.236]{K-L}).
\par
For a nontrivial character $\chi$ of $\mathrm{Cl}(\frak{f})$ with
$\mathfrak{f}\neq\mathcal{O}_K$, we define the \textit{Stickelberger
element} as
\begin{equation*}
S_\mathfrak{f}(\chi,g_\mathfrak{f})
=\sum_{C\in\mathrm{Cl}(\mathfrak{f})}
\chi(C)\log|g_\mathfrak{f}(C)|,
\end{equation*}
and consider the \textit{$L$-function}
\begin{equation*}
L_\mathfrak{f}(s,\chi)=\sum_{\mathfrak{a}\neq0~:~\textrm{integral
ideals of
$K$}}\frac{\chi(\mathfrak{a})}{\mathbf{N}_{K/\mathbb{Q}}(\mathfrak{a})^s}
\quad(s\in\mathbb{C}).
\end{equation*}
From the second Kronecker limit formula (\cite[Chapter 22 Theorem
1]{Lang}) we get the following proposition.

\begin{proposition}\label{LandS}
The notation being as above, if
$\mathfrak{f}_\chi\neq\mathcal{O}_K$, then
\begin{equation*}
\prod_{\mathfrak{p}|\mathfrak{f},\mathfrak{p}\nmid\mathfrak{f}_\chi}
(1-\overline{\chi}_0(\mathfrak{p})) L_{\mathfrak{f}_\chi}(1,\chi_0)=
\frac{\pi}{3\omega(\mathfrak{f})N(\mathfrak{f})\tau(\overline{\chi}_0)\sqrt{{|d_K|}}}
S_{\mathfrak{f}}(\overline{\chi},g_{\mathfrak{f}}),
\end{equation*}
where
\begin{eqnarray*}
\tau(\overline{\chi}_0)=~-\hspace{-1cm}
\sum_{\begin{smallmatrix}x\in\mathcal{O}_K\\x\hspace{-0.2cm}\mod{\mathfrak{f}}\\\gcd(x\mathcal{O}_K,\mathfrak{f}_\chi)=\mathcal{O}_K\end{smallmatrix}}
\hspace{-0.5cm}\overline{\chi}_0([x\gamma\mathfrak{d}_K\mathfrak{f}_\chi])e^{2\pi
i\mathbf{Tr}_{K/\mathbb{Q}}(x\gamma)}
\end{eqnarray*}
with $\gamma$ any element of $K$ such that
$\gamma\mathfrak{d}_K\mathfrak{f}_\chi$ is an integral ideal
relatively prime to $\mathfrak{f}$.
\end{proposition}
\begin{proof}
See \cite[Chapter 22 Theorem 2]{Lang} and \cite[Chapter 11 Theorem
2.1]{K-L}.
\end{proof}

\begin{remark}\label{remarkLandS}
\begin{itemize}
\item[(i)] The product factor
$\prod_{\mathfrak{p}|\mathfrak{f},\mathfrak{p}\nmid\mathfrak{f}_\chi}
(1-\overline{\chi}_0(\mathfrak{p}))$ is called the \textit{Euler
factor of $\chi$}. If there is no such $\mathfrak{p}$ with
$\mathfrak{p}|\mathfrak{f}$ and
$\mathfrak{p}\nmid\mathfrak{f}_\chi$, then it is understood to be
$1$.
\item[(ii)] As is well-known, $L_{\mathfrak{f}_\chi}(1,\chi_0)\neq0$
(\cite[Chapter IV Proposition 5.7]{Janusz}).
\end{itemize}
\end{remark}

\begin{lemma}\label{characterextend}
Let $A\varsubsetneq B$ be finite abelian groups, $b\in B-A$ and
$\chi$ be a character of $A$. Let $m$ be the order of the coset $bA$
in the quotient group $B/A$. Then we can extend $\chi$ to a
character $\psi$ of $B$ such that $\psi(b)$ is any $m^\textrm{th}$
root of $\chi(b^m)$.
\end{lemma}
\begin{proof}
It suffices to prove the case $B=\langle A,b\rangle$. Let $\zeta$ be
any $m^\textrm{th}$ root of $\chi(b^m)$. Define a map
\begin{eqnarray*}
\psi~:~\langle A,b\rangle&\longrightarrow&\mathbb{C}^*\\
ab^k&\mapsto&\chi(a)\zeta^k\quad(a\in A).
\end{eqnarray*}
Using the fact $\zeta^m=\chi(b^m)$ one can readily show that $\psi$
is a well-defined character of $\langle A,b\rangle$ which extends
$\chi$ and also satisfies $\psi(b)=\zeta$.
\end{proof}

\begin{lemma}\label{character}
Let $K$ be an imaginary quadratic field and
$\mathfrak{f}=\prod_{k=1}^n\mathfrak{p}_k^{e_k}$ be a nontrivial
ideal of $K$. Let $L$ be a finite abelian extension of $K$ such that
$K\subsetneq L \subseteq K_\mathfrak{f}$. For an intermediate field
$F$ between $K$ and $K_\mathfrak{f}$ we denote by
$\mathrm{Cl}(K_\mathfrak{f}/F)$ the subgroup of
$\mathrm{Cl}(\mathfrak{f})$ corresponding to
$\mathrm{Gal}(K_\mathfrak{f}/F)$ via the Artin map. Let
\begin{eqnarray*}
\widehat{\varepsilon}_k&=&\#~\mathrm{Ker}(\textrm{the natural
projection}~\widehat{\rho}_k~:~\pi_\mathfrak{f}(\mathcal{O}_K)^*/\pi_\mathfrak{f}(\mathcal{O}_K^*)
\rightarrow
\pi_{\mathfrak{f}\mathfrak{p}_k^{-e_k}}(\mathcal{O}_K)^*/\pi_{\mathfrak{f}\mathfrak{p}_k^{-e_k}}(\mathcal{O}_K^*)),\\
\varepsilon_k&=&\#~\mathrm{Ker}(\textrm{the natural
projection}~\rho_k~:~\pi_\mathfrak{f}(\mathcal{O}_K)^*/\pi_\mathfrak{f}(\mathcal{O}_K^*)
\longrightarrow
\pi_{\mathfrak{p}_k^{e_k}}(\mathcal{O}_K)^*/\pi_{\mathfrak{p}_k^{e_k}}(\mathcal{O}_K^*))
\end{eqnarray*}
for each $k=1,\cdots,n$. Assume that
\begin{equation}\label{hypothesis}
\begin{array}{l}
\textrm{for each $k=1,\cdots,n$ there is an odd prime $\nu_k$ such that}\\
\textrm{$\nu_k\nmid\varepsilon_k$ and
$\mathrm{ord}_{\nu_k}(\widehat{\varepsilon}_k)>
\mathrm{ord}_{\nu_k}(\#~\mathrm{Cl}(K_\mathfrak{f}/L))$.}
\end{array}
\end{equation}
If $D$ is a class in
$\mathrm{Cl}(\mathfrak{f})-\mathrm{Cl}(K_\mathfrak{f}/L)$, then
there exists a character $\chi$ of $\mathrm{Cl}(\mathfrak{f})$ such
that
\begin{equation}\label{conditions}
\chi|_{\mathrm{Cl}(K_\mathfrak{f}/L)}=1,~\chi(D)\neq1~ \textrm{and}~
\mathfrak{p}_k|\mathfrak{f}_\chi~(k=1,\cdots, n).
\end{equation}
\end{lemma}
\begin{proof}
Since $D\in\mathrm{Cl}(\mathfrak{f})-\mathrm{Cl}(K_\mathfrak{f}/L)$,
there is a character $\chi$ of $\mathrm{Cl}(\mathfrak{f})$ such that
\begin{equation*}
\chi|_{\mathrm{Cl}(K_\mathfrak{f}/L)}=1~\textrm{and}~\chi(D)\neq1
\end{equation*}
by Lemma \ref{characterextend}. Let $\widetilde{\chi}_k$
($k=1,\cdots,n$) be the character of
$\pi_{\mathfrak{p}_k^{e_k}}(\mathcal{O}_K)^*$ induced from $\chi$ as
in Lemma \ref{conductor}.
\par
Suppose $\widetilde{\chi}_k=1$ for some $k$. Let $\nu_k$ be a prime
number in the assumption (\ref{hypothesis}) and $S$ be a Sylow
$\nu_k$-subgroup of
$\Phi_\mathfrak{f}(\mathrm{Ker}(\widehat{\rho}_k))$. Then
$\mathrm{Cl}(K_\mathfrak{f}/L)$ does not contain $S$ by
(\ref{hypothesis}). Hence we can take an element $C$ in
$S-\mathrm{Cl}(K_\mathfrak{f}/L)$ whose order is a power of $\nu_k$.
Now we extend the trivial character of
$\mathrm{Cl}(K_\mathfrak{f}/L)$ to a character $\psi'$ of
$\mathrm{Cl}(\mathfrak{f})$ so that $\psi'(C)=\zeta_{\nu_k}=e^{2\pi
i/\nu_k}$ by Lemma \ref{characterextend}, because the order of the
coset $C\mathrm{Cl}(K_\mathfrak{f}/L)$ in the quotient group
$\mathrm{Cl}(\mathfrak{f})/\mathrm{Cl}(K_\mathfrak{f}/L)$ is also a
power of $\nu_k$. Define a character $\psi$ of
$\mathrm{Cl}(\mathfrak{f})$ by
\begin{equation*}
\psi=\left\{\begin{array}{ll} \psi'^{\varepsilon_k} & \textrm{if
$\chi(D)\psi'^{\varepsilon_k}(D)\neq1$}\vspace{0.1cm}\\
\psi'^{2\varepsilon_k} & \textrm{otherwise.}
\end{array}\right.
\end{equation*}
We then achieve $(\chi\psi)|_{\mathrm{Cl}(K_\mathfrak{f}/L)}=1$ and
$(\chi\psi)(D)=\chi(D)\psi(D)\neq1$. Furthermore, since
$(\iota\circ\iota_\ell(\pi_{\mathfrak{p}_\ell^{e_\ell}}(\mathcal{O}_K)^*))
\pi_\mathfrak{f}(\mathcal{O}_K^*)/
\pi_\mathfrak{f}(\mathcal{O}_K^*)$ is a subgroup of
$\mathrm{Ker}(\rho_k)$ for $\ell\neq k$ (see the diagram
(\ref{diagram2}) below), we derive that
\begin{eqnarray*}
\widetilde{\psi}_\ell(\pi_{\mathfrak{p}_\ell^{e_\ell}}(\mathcal{O}_K)^*)&=&
\psi\circ\widetilde{\Phi}_\mathfrak{f}\circ\iota\circ\iota_\ell(\pi_{\mathfrak{p}_\ell^{e_\ell}}(\mathcal{O}_K)^*)
\quad\textrm{by the definition of $\widetilde{\psi}_\ell$ in Lemma \ref{conductor}}\\
&\subseteq&\psi(\Phi_\mathfrak{f}(\mathrm{Ker}(\rho_k)))\\
&=&\psi'^{\varepsilon_k}(\Phi_\mathfrak{f}(\mathrm{Ker}(\rho_k)))
~\textrm{or}~\psi'^{2\varepsilon_k}(\Phi_\mathfrak{f}(\mathrm{Ker}(\rho_k)))=1,
\end{eqnarray*}
which yields $\widetilde{\psi}_\ell=1$ and
$(\widetilde{\chi\psi})_\ell=\widetilde{\chi}_\ell\widetilde{\psi}_\ell=\widetilde{\chi}_\ell$
for $\ell\neq k$. On the other hand, since
$C\in\Phi_\mathfrak{f}(\mathrm{Ker}(\widehat{\rho}_k))\subseteq\mathrm{Im}(\Phi_\mathfrak{f})$,
we can take an element $c$ of $\pi_\mathfrak{f}(\mathcal{O}_K)^*$ so
that $\widetilde{\Phi}_\mathfrak{f}(c)=C$. Thus we get
\begin{eqnarray*}
\widetilde{\psi}(c)=\psi\circ\widetilde{\Phi}_\mathfrak{f}(c)=\psi(C)
=(\psi'^{\varepsilon_k}(C)~\textrm{or}~\psi'^{2\varepsilon_k}(C))=(\zeta_{\nu_k}^{\varepsilon_k}
~\textrm{or}~\zeta_{\nu_k}^{2\varepsilon_k})\neq1,
\end{eqnarray*}
which shows $\widetilde{\psi}\neq1$. Hence $\widetilde{\psi}_k\neq1$
by the fact $\widetilde{\psi}_\ell=1$ for $\ell\neq k$ and Lemma
\ref{conductor}(ii). Therefore we obtain
$(\widetilde{\chi\psi})_k=\widetilde{\chi}_k\widetilde{\psi}_k=\widetilde{\psi}_k\neq1$.
\par
Now, we replace $\chi$ by $\chi\psi$ and repeat the above process
for finitely many $\ell$ ($\neq k$) such that
$\widetilde{\chi}_\ell=1$. After this procedure we finally establish
a character $\chi$ of $\mathrm{Cl}(\mathfrak{f})$ which satisfies
\begin{equation*}
\chi|_{\mathrm{Cl}(K_\mathfrak{f}/L)}=1,~ \chi(D)\neq1~\textrm{and}~
\widetilde{\chi}_k\neq1~(k=1,\cdots, n).
\end{equation*}
We derive by Lemma \ref{conductor} that
$\mathfrak{p}_k|\mathfrak{f}_\chi$ for all $k=1,\cdots,n$. This
proves the lemma.
\end{proof}

\begin{remark}\label{e_kcondition}
From the commutative diagram of exact sequences
\begin{equation*}
\begindc{\commdiag}[50]
\obj(1,2)[A]{$1$}
\obj(3,2)[B]{$\pi_\mathfrak{f}(\mathcal{O}_K)^*/\pi_\mathfrak{f}(\mathcal{O}_K^*)$}
\obj(6,2)[C]{$\mathrm{Cl}(\mathfrak{f})$}
\obj(8,2)[D]{$\mathrm{Cl}(\mathcal{O}_K)$} \obj(9,2)[E]{$1$}
\obj(1,1)[F]{$1$}
\obj(3,1)[G]{$\pi_{\mathfrak{f}\mathfrak{p}_k^{-e_k}}(\mathcal{O}_K)^*/\pi_{\mathfrak{f}\mathfrak{p}_k^{-e_k}}(\mathcal{O}_K^*)$}
\obj(6,1)[H]{$\mathrm{Cl}(\mathfrak{f}\mathfrak{p}_k^{-e_k})$}
\obj(8,1)[I]{$\mathrm{Cl}(\mathcal{O}_K)$} \obj(9,1)[J]{$1$}
\mor{A}{B}{} \mor{B}{C}{$\Phi_\mathfrak{f}$} \mor{C}{D}{}
\mor{D}{E}{} \mor{F}{G}{}
\mor{G}{H}{$\Phi_{\mathfrak{f}\mathfrak{p}_k^{-e_k}}$} \mor{H}{I}{}
\mor{I}{J}{} \mor{B}{G}{$\widehat{\rho}_k$} \mor{C}{H}{}
\mor{D}{I}{}
\enddc
\end{equation*}
where vertical maps are natural projections, one can readily obtain
\begin{equation*}
\mathrm{Cl}(\mathfrak{f})/\Phi_\mathfrak{f}(\mathrm{Ker}(\widehat{\rho}_k))\simeq
\mathrm{Cl}(\mathfrak{f}\mathfrak{p}_k^{-e_k})\simeq
\mathrm{Cl}(\mathfrak{f})/\mathrm{Cl}(K_\mathfrak{f}/K_{\mathfrak{f}\mathfrak{p}_k^{-e_k}}).
\end{equation*}
Hence we have
\begin{equation*}
\widehat{\varepsilon}_k=\#~\mathrm{Ker}(\widehat{\rho}_k)=
\#~\Phi_\mathfrak{f}(\mathrm{Ker}(\widehat{\rho}_k))=
[K_\mathfrak{f}:K_{\mathfrak{f}\mathfrak{p}_k^{-e_k}}]
=\varphi(\mathfrak{p}_k^{e_k})\omega(\mathfrak{f})/\omega(\mathfrak{f}\mathfrak{p}_k^{-e_k})
\end{equation*}
by using Lemma \ref{degree}(ii), which will be used in the next
section. Similarly, again from the commutative diagram
\begin{equation}\label{diagram2}
\begindc{\commdiag}[50]
\obj(1,3)[A]{$1$}
\obj(3,3)[B]{$\pi_\mathfrak{f}(\mathcal{O}_K)^*/\pi_\mathfrak{f}(\mathcal{O}_K^*)$}
\obj(6,3)[C]{$\mathrm{Cl}(\mathfrak{f})$}
\obj(8,3)[D]{$\mathrm{Cl}(\mathcal{O}_K)$} \obj(9,3)[E]{$1$}
\obj(1,1)[F]{$1$}
\obj(3,1)[G]{$\pi_{\mathfrak{p}_k^{e_k}}(\mathcal{O}_K)^*/\pi_{\mathfrak{p}_k^{e_k}}(\mathcal{O}_K^*)$}
\obj(6,1)[H]{$\mathrm{Cl}(\mathfrak{p}_k^{e_k})$}
\obj(8,1)[I]{$\mathrm{Cl}(\mathcal{O}_K)$} \obj(9,1)[J]{$1$}
\obj(3,2)[K]{$\prod_{\ell=1}^n
\pi_{\mathfrak{p}_\ell^{e_\ell}}(\mathcal{O}_K)^*/\{\prod_{\ell=1}^n\pi_{\mathfrak{p}_\ell^{e_\ell}}(x):x\in\mathcal{O}_K^*\}$}
\mor{A}{B}{} \mor{B}{C}{$\Phi_\mathfrak{f}$}
\mor{B}{K}{$\wr$}[\atleft, \solidline] \mor{C}{D}{} \mor{D}{E}{}
\mor{F}{G}{} \mor{G}{H}{$\Phi_{\mathfrak{p}_k^{e_k}}$} \mor{H}{I}{}
\mor{I}{J}{} \mor{K}{G}{$\rho_k$} \mor{C}{H}{} \mor{D}{I}{}
\enddc
\end{equation}
we come up with
\begin{equation*}
\varepsilon_k=\#~\mathrm{Ker}(\rho_k)
=\#~\Phi_\mathfrak{f}(\mathrm{Ker}(\rho_k))
=[K_\mathfrak{f}:K_{\mathfrak{p}_k^{e_k}}]=
\frac{\prod_{\ell=1}^n\varphi(\mathfrak{p}_\ell^{e_\ell})\omega(\mathfrak{f})}
{\varphi(\mathfrak{p}_k^{e_k})\omega(\mathfrak{p}_k^{e_k})}.
\end{equation*}
\end{remark}

\begin{theorem}\label{generator}
Let $L$ be a field in Lemma \textup{\ref{character}} which satisfies
the assumption \textup{(\ref{hypothesis})}. Then the singular value
\begin{equation*}
\varepsilon=\mathbf{N}_{K_\mathfrak{f}/L}(g_\mathfrak{f}(C_0))
\end{equation*}
generates $L$ over $K$.
\end{theorem}
\begin{proof}
Let $F=K(\varepsilon)$ as a subfield of $L$. Suppose that $F$ is
properly contained in $L$. Then for a class $D$ in
$\mathrm{Cl}(K_\mathfrak{f}/F)-\mathrm{Cl}(K_\mathfrak{f}/L)$ we can
find a character $\chi$ of $\mathrm{Cl}(\mathfrak{f})$ satisfying
the conditions (\ref{conditions}) in Lemma \ref{character}. Since
the Euler factor of $\chi$ is $1$ by the condition
$\mathfrak{p}_k|\mathfrak{f}_\chi$ for all $k$, the value
$S_\mathfrak{f}(\overline{\chi},g_\mathfrak{f})$ does not vanish by
Proposition \ref{LandS} and Remark \ref{remarkLandS}(ii). On the
other hand,
\begin{eqnarray*}
S_\mathfrak{f}(\overline{\chi},g_\mathfrak{f})
&=&\sum_{\begin{smallmatrix}C_1\in\mathrm{Cl}(\mathfrak{f})\\
C_1~\mathrm{mod}~\mathrm{Cl}(K_\mathfrak{f}/F)\end{smallmatrix}}
\sum_{\begin{smallmatrix}C_2\in\mathrm{Cl}(K_\mathfrak{f}/F)\\
C_2~\mathrm{mod}~\mathrm{Cl}(K_\mathfrak{f}/L)\end{smallmatrix}}
\sum_{C_3\in\mathrm{Cl}(K_\mathfrak{f}/L)}\overline{\chi}(C_1C_2C_3)\log|g_\mathfrak{f}(C_1C_2C_3)|\\
&=&\sum_{C_1}\overline{\chi}(C_1) \sum_{C_2} \overline{\chi}(C_2)
\sum_{C_3}\log|g_\mathfrak{f}(C_0)^{\sigma(C_1)\sigma(C_2)\sigma(C_3)}|\quad\textrm{by
$\chi|_{\mathrm{Cl}(K_\mathfrak{f}/L)}=1$ and
(\ref{Artin})}\\
&=&\sum_{C_1}\overline{\chi}(C_1) \sum_{C_2} \overline{\chi}(C_2)
\log|\varepsilon^{\sigma(C_1)\sigma(C_2)}|\\
&=&\sum_{C_1}\overline{\chi}(C_1) (\sum_{C_2} \overline{\chi}(C_2))
\log|\varepsilon^{\sigma(C_1)}|\quad\textrm{by the fact
$\varepsilon\in
F$}\\
&=&0\quad\textrm{because $\chi(D)\neq1$ implies
$\chi|_{\mathrm{Cl}(K_\mathfrak{f}/F)}\neq1$,}
\end{eqnarray*}
which gives a contradiction. Therefore $L=F$, as desired.
\end{proof}

\begin{remark}\label{remarkgenerator}
Observe that any nonzero power of $\varepsilon$ generates $L$ over
$K$, too.
\end{remark}

\section{Actions of Galois groups}

In this section we shall determine Galois groups of ray class fields
over ring class fields by Shimura's reciprocity law.
\par
For an integer $N$ ($\geq1$) let $\zeta_N=e^{2\pi i/N}$ and
$\Gamma(N)=\{\gamma\in\mathrm{SL}_2(\mathbb{Z}):\gamma\equiv\left(\begin{smallmatrix}
1&0\\0&1\end{smallmatrix}\right)\pmod{N}\}$. Furthermore, we let
$\mathcal{F}_N$ be the field of modular functions for $\Gamma(N)$
whose Fourier coefficients lie in $\mathbb{Q}(\zeta_N)$.

\begin{proposition}\label{Gal(F_N/F_1)}
$\mathcal{F}_N$ is a Galois extension of
$\mathcal{F}_1=\mathbb{Q}(j(\tau))$ whose Galois group is isomorphic
to
\begin{equation*}
\mathrm{GL}_2(\mathbb{Z}/N\mathbb{Z})/\{\pm1_2\}=G_N\cdot
\mathrm{SL}_2(\mathbb{Z}/N\mathbb{Z})/\{\pm1_2\}=
\mathrm{SL}_2(\mathbb{Z}/N\mathbb{Z})/\{\pm1_2\}\cdot G_N,
\end{equation*}
where
$G_N=\{\left(\begin{smallmatrix}1&0\\0&d\end{smallmatrix}\right)
:d\in(\mathbb{Z}/N\mathbb{Z})^*\}$. Here, the matrix
$\left(\begin{smallmatrix}1&0\\0&d\end{smallmatrix}\right)\in G_N$
acts on $\sum_{n>-\infty}c_n q^{n/N}\in\mathcal{F}_N$ by
\begin{equation*}
\sum_{n>-\infty} c_nq^{n/N}\mapsto \sum_{n>-\infty}
c_n^{\sigma_d}q^{n/N},
\end{equation*}
where $\sigma_d$ is the automorphism of $\mathbb{Q}(\zeta_N)$
induced by $\zeta_N\mapsto\zeta_N^d$. And, for an element
$\gamma\in\mathrm{SL}_2(\mathbb{Z}/N\mathbb{Z})/\{\pm1_2\}$ let
$\gamma'\in\mathrm{SL}_2(\mathbb{Z})$ be a preimage of $\gamma$ via
the natural surjection
$\mathrm{SL}_2(\mathbb{Z})\rightarrow\mathrm{SL}_2(\mathbb{Z}/N\mathbb{Z})/\{\pm1_2\}$.
Then $\gamma$ acts on $h\in\mathcal{F}_N$ by composition
\begin{equation*}
h\mapsto h\circ\gamma'
\end{equation*}
as a fractional linear transformation.
\end{proposition}
\begin{proof}
See \cite[Chapter 6 Theorem 3]{Lang}.
\end{proof}

\begin{proposition}\label{functionj}
Let $N$ be a positive integer.
\begin{itemize}
\item[(i)] The fixed field of $\mathcal{F}_N$ by $\Gamma_0(N)$ is
the field $\mathbb{Q}(j(\tau),j(N\tau),\zeta_N)$.
\item[(ii)] $j(N\tau)$ has rational Fourier coefficients.
\item[(iii)] $\Delta(N\tau)/\Delta(\tau)$ belongs to $\mathcal{F}_N$
and has rational Fourier coefficients.
\end{itemize}
\end{proposition}
\begin{proof}
(i) See \cite[Chapter 6 Theorem 7]{Lang}.\\
(ii) See \cite[Chapter 4 $\S$1]{Lang}.\\
(iii) See \cite[Chapter 11 Theorem 4]{Lang}.
\end{proof}

We need some transformation formulas of Siegel functions to apply
the above proposition.

\begin{proposition}\label{TransformSiegel}
Let $(r_1,r_2)\in(1/N)\mathbb{Z}^2-\mathbb{Z}^2$ for $N\geq2$.
\begin{itemize}
\item[(i)] $g_{(r_1,r_2)}(\tau)^{12N}$ satisfies
\begin{equation*}
g_{(r_1,r_2)}(\tau)^{12N}=g_{(-r_1,-r_2)}(\tau)^{12N} =g_{(\langle
r_1\rangle,\langle r_2\rangle)}(\tau)^{12N},
\end{equation*}
where $\langle X\rangle$ is the fractional part of $X\in\mathbb{R}$
such that $0\leq \langle X\rangle<1$.
\item[(ii)] $g_{(r_1,r_2)}(\tau)^{12N}$ belongs to $\mathcal{F}_N$,
and $\alpha\in \mathrm{GL}_2(\mathbb{Z}/N\mathbb{Z})/\{\pm1_2\}$
\textup{(}$\simeq\mathrm{Gal}(\mathcal{F}_N/
\mathcal{F}_1)$\textup{) acts on the function by}
\begin{equation*}
(g_{(r_1,r_2)}(\tau)^{12N})^\alpha= g_{(r_1,r_2)\alpha}(\tau)^{12N}.
\end{equation*}
\item[(iii)] $g_{(r_1,r_2)}(\tau)$ is integral over
$\mathbb{Z}[j(\tau)]$.
\end{itemize}
\end{proposition}
\begin{proof}
(i) See \cite[Proposition 2.4(1), (3)]{K-S}.\\
(ii) See \cite[Chapter 2 Proposition 1.3]{K-L}.\\
(iii) See \cite[$\S$3]{K-S}.
\end{proof}

Let $K$ be an imaginary quadratic field of discriminant $d_K$, and
define
\begin{eqnarray}\label{theta}
\theta=\left\{\begin{array}{ll}\sqrt{d_K}/2&\textrm{for}~d_K\equiv0\pmod{4}\vspace{0.1cm}\\
(-1+\sqrt{d_K})/2&\textrm{for}~
d_K\equiv1\pmod{4},\end{array}\right.
\end{eqnarray}
from which we get $\mathcal{O}_K=\mathbb{Z}[\theta]$. We see from
the main theorem of the theory of complex multiplication that for
every positive integer $N$,
\begin{equation*}
K_{(N)}=K\mathcal{F}_N(\theta)=K(h(\theta)~:~h\in\mathcal{F}_N~\textrm{is
defined and finite at $\theta$})
\end{equation*}
(\cite[Chapter 10 Corollary to Theorem 2]{Lang}).
\par
Let
\begin{equation*}
\mathrm{min}(\theta,\mathbb{Q})=X^2+B_\theta
X+C_\theta=\left\{\begin{array}{ll} X^2-d_K/4&
\textrm{if}~d_K\equiv0\pmod{4}\vspace{0.1cm}\\
X^2+X+(1-d_K)/4 & \textrm{if}~d_K\equiv1\pmod{4}.
\end{array}\right.
\end{equation*}
For every positive integer $N$, we define the matrix group
\begin{equation*}
W_{N,\theta}=\bigg\{\begin{pmatrix}t-B_\theta s & -C_\theta
s\\s&t\end{pmatrix}\in\mathrm{GL}_2(\mathbb{Z}/N\mathbb{Z})~:~t,s\in\mathbb{Z}/N\mathbb{Z}\bigg\}.
\end{equation*}
Due to Stevenhagen we have the following explicit description of
Shimura's reciprocity law (\cite[Theorem 6.31 and Proposition
6.34]{Shimura}), which relates the class field theory to the theory
of modular functions.

\begin{proposition}\label{reciprocity}
For each positive integer $N$, the matrix group $W_{N,\theta}$ gives
rise to the surjection
\begin{equation}\begin{array}{ccl}\label{surj}
W_{N,\theta}&\longrightarrow&\mathrm{Gal}(K_{(N)}/H)\vspace{0.1cm}\\
\alpha&\mapsto&(h(\theta)\mapsto
h^\alpha(\theta)~:~\textrm{$h\in\mathcal{F}_N$ is defined and finite
at $\theta$}),\end{array}
\end{equation}
whose kernel is
\begin{equation*}
\left\{\begin{array}{ll}
\bigg\{\pm\begin{pmatrix}1&0\\0&1\end{pmatrix},
\pm\begin{pmatrix}0&-1\\1&0\end{pmatrix}\bigg\} & \textrm{if $K=\mathbb{Q}(\sqrt{-1})$}\vspace{0.1cm}\\
\bigg\{\pm\begin{pmatrix}1&0\\0&1\end{pmatrix},
\pm\begin{pmatrix}-1&-1\\1&0\end{pmatrix},
\pm\begin{pmatrix}0&-1\\1&1\end{pmatrix}\bigg\} & \textrm{if $K=\mathbb{Q}(\sqrt{-3})$}\vspace{0.1cm}\\
\bigg\{\pm\begin{pmatrix}1&0\\0&1\end{pmatrix}\bigg\} &
\textrm{otherwise.}
\end{array}\right.
\end{equation*}
\end{proposition}
\begin{proof}
See \cite[$\S$3]{Stevenhagen}.
\end{proof}

The \textit{ring class field $H_\mathcal{O}$ of the order
$\mathcal{O}$ of conductor $N$ \textup{(}$\geq1$\textup{)} in $K$}
is by the definition a finite abelian extension of $K$ whose Galois
group is isomorphic to
$I_K(N\mathcal{O}_K)/P_{K,\mathbb{Z}}(N\mathcal{O}_K)$ via the Artin
map, where $P_{K,\mathbb{Z}}(N\mathcal{O}_K)$ is the subgroup of
$P_K(N\mathcal{O}_K)$ generated by principal ideals
$\alpha\mathcal{O}_K$ with $\alpha\equiv a\pmod{N\mathcal{O}_K}$ for
some integer $a$ prime to $N$. Then, $H_\mathcal{O}$ is contained in
the ray class field $K_{(N)}$.

\begin{proposition}\label{jlemma}
Let $K$ be an imaginary quadratic field and $\theta$ be as in
\textup{(\ref{theta})}. Let $\mathcal{O}$ be the order of conductor
$N$ \textup{(}$\geq1$\textup{)} in $K$.
\begin{itemize}
\item[(i)] $j(\mathcal{O})=j(N\theta)$ is an algebraic integer
which generates $H_\mathcal{O}$ over $K$.
\item[(ii)] $\Delta(N\theta)/\Delta(\theta)$ is a real algebraic number
lying in $H_\mathcal{O}$.
\end{itemize}
\end{proposition}
\begin{proof}
(i) See \cite[Chapter 5 Theorem 4 and Chapter 10 Theorem 5]{Lang}.\\
(ii) See \cite[Chapter 12 Corollary to Theorem 1]{Lang}.
\end{proof}

\begin{lemma}\label{jfix}
Let $K$ be an imaginary quadratic field and $\theta$ be as in
\textup{(\ref{theta})}. Let $N$ be a positive integer. Then,
$\left(\begin{smallmatrix}t&0\\0&t\end{smallmatrix}\right)\in
W_{N,\theta}$ fixes $j(N\theta)$.
\end{lemma}
\begin{proof}
Decompose
$\left(\begin{smallmatrix}t&0\\0&t\end{smallmatrix}\right)\in
W_{N,\theta}$ into
$\left(\begin{smallmatrix}t&0\\0&t\end{smallmatrix}\right)=
\left(\begin{smallmatrix}1&0\\0&t^2\end{smallmatrix}\right)\cdot\alpha
\in G_N\cdot\mathrm{SL}_2(\mathbb{Z}/N\mathbb{Z})/\{\pm1_2\}$ as in
Proposition \ref{Gal(F_N/F_1)}, and let $\alpha'$ be a preimage of
$\alpha$ via the natural surjection
$\mathrm{SL}_2(\mathbb{Z})\rightarrow\mathrm{SL}_2(\mathbb{Z}/N\mathbb{Z})/\{\pm1_2\}$.
Then, $\alpha'$ belongs to $\Gamma_0(N)$. We then obtain from
Propositions \ref{Gal(F_N/F_1)} and \ref{reciprocity} that
\begin{eqnarray*}
j(N\theta)^{\left(\begin{smallmatrix}t&0\\0&t\end{smallmatrix}\right)}&=&
j(N\tau)^{\left(\begin{smallmatrix}t&0\\0&t\end{smallmatrix}\right)}(\theta)
=j(N\tau)^{\left(\begin{smallmatrix}1&0\\0&t^2\end{smallmatrix}\right)\alpha}(\theta)\\
&=&j(N\tau)^{\alpha}(\theta)\quad\textrm{by Proposition \ref{functionj}(ii)}\\
&=&j(N\tau)\circ\alpha'(\theta)\\
&=&j(N\theta)\quad\textrm{by the fact $\alpha'\in\Gamma_0(N)$ and
Proposition \ref{functionj}(i).}
\end{eqnarray*}
This proves the lemma.
\end{proof}

\begin{lemma}\label{degree}
Let $K$ be an imaginary quadratic field of discriminant $d_K$. We
have the following degree formulas:
\begin{itemize}
\item[(i)] If $\mathcal{O}$ is the order of conductor $N$ \textup{(}$\geq1$\textup{)} in
$K$, then
\begin{equation*}
[H_\mathcal{O}:K]=\frac{h_K
N}{(\mathcal{O}_K^*:\mathcal{O}^*)}\prod_{p|N}\bigg(1-\bigg(\frac{d_K}{p}\bigg)\frac{1}{p}\bigg),
\end{equation*}
where $h_K$ is the class number of $K$ and $(d_K/p)$ is the
Kronecker symbol.
\item[(ii)] If $\mathfrak{f}$ is a nonzero integral ideal of $K$, then
\begin{equation*}
[K_{\mathfrak{f}}:K]={h_K\varphi(\mathfrak{f})\omega(\mathfrak{f})}/{\omega_K},
\end{equation*}
where $\varphi$ is the Euler function for ideals, namely
\begin{equation*}
\varphi(\mathfrak{p}^n)=(\mathbf{N}_{K/\mathbb{Q}}(\mathfrak{p})-1)\mathbf{N}_{K/\mathbb{Q}}(\mathfrak{p})^{n-1}
\end{equation*}
for a power of prime ideal $\mathfrak{p}$ \textup{(}and we set
$\varphi(\mathcal{O}_K)=1$\textup{)}.
\end{itemize}
\end{lemma}
\begin{proof}
(i) See \cite[Chapter 8 Theorem 7]{Lang}.\\
(ii) See \cite[Chapter VI Theorem 1]{Lang2}.
\end{proof}

\begin{proposition}\label{Gal(K_N/H_O)}
Let $\mathcal{O}$ be the order of conductor $N$
\textup{(}$\geq1$\textup{)} in an imaginary quadratic field $K$. The
map in \textup{(\ref{surj})} induces an isomorphism
\begin{equation*}
\bigg\{\begin{pmatrix} t &0\\0&t
\end{pmatrix}~:~t\in(\mathbb{Z}/N\mathbb{Z})^*\bigg\}\bigg/\bigg\{\pm
\begin{pmatrix}1&0\\0&1\end{pmatrix}\bigg\}
\stackrel{\sim}{\longrightarrow}\mathrm{Gal}(K_{(N)}/H_\mathcal{O}).
\end{equation*}
\end{proposition}
\begin{proof}
If $N=1$, then it is obvious. So, let $N\geq2$. Observe first that
the above map is well-defined and injective by Proposition
\ref{reciprocity} and Lemma \ref{jfix}. Let $N=\prod_{a=1}^A
p_a^{u_a} \prod_{b=1}^B q_b^{v_b} \prod_{c=1}^C r_c^{w_c}$ be the
prime factorization of $N$, where each $p_a$ (respectively, $q_b$
and $r_c$) splits (respectively, is inert and ramified) in
$K/\mathbb{Q}$. (We understand $\prod_{1}^0$ as $1$.) Note that
\begin{equation}\label{Legendre}
(d_K/p_a)=1,~(d_K/q_b)=-1,~(d_K/r_c)=0,
\end{equation}
and we have the prime ideal factorization
$N\mathcal{O}_K=\prod_{a=1}^A
(\mathfrak{p}_a\overline{\mathfrak{p}}_a)^{u_a} \prod_{b=1}^B
\mathfrak{q}_b^{v_b} \prod_{c=1}^C \mathfrak{r}_c^{2w_c}$ with
\begin{equation}\label{norm}
\mathbf{N}_{K/\mathbb{Q}}(\mathfrak{p}_a)=
\mathbf{N}_{K/\mathbb{Q}}(\overline{\mathfrak{p}}_a)= p_a,~
\mathbf{N}_{K/\mathbb{Q}}(\mathfrak{q}_b)=q_b^2,~
\mathbf{N}_{K/\mathbb{Q}}(\mathfrak{r}_c)=r_c.
\end{equation}
We derive by Lemma \ref{degree} that
\begin{eqnarray*}
&&\#~\mathrm{Gal}(K_{(N)}/H_\mathcal{O})=[K_{(N)}:H_\mathcal{O}]=\frac{[K_{(N)}:K]}{[H_\mathcal{O}:K]}\\
&=&\frac{\varphi(N\mathcal{O}_K)\omega(N\mathcal{O}_K)}{2N\prod_{p|N}(1-(\frac{d_K}{p})\frac{1}{p})}
\quad\textrm{by the facts $\omega_K=\#~\mathcal{O}_K^*$ and $\mathcal{O}^*=\{\pm1\}$}\\
&=&\frac{\omega(N\mathcal{O}_K)}{2}\frac{\prod_{a=1}^A((p_a-1)p_a^{u_a-1})^2
\prod_{b=1}^B(q_b^2-1)q_b^{2(v_b-1)}
\prod_{c=1}^C(r_c-1)r_c^{2w_c-1}} {\prod_{a=1}^A
p_a^{u_a-1}(p_a-1)\prod_{b=1}^B q_b^{v_b-1}(q_b+1)\prod_{c=1}^C
r_c^{w_c}}\\
&&\textrm{by (\ref{Legendre}) and (\ref{norm})}\\
&=&\frac{\omega(N\mathcal{O}_K)}{2}\prod_{a=1}^A(p_a-1)p^{u_a-1}
\prod_{b=1}^B(q_b-1)q_b^{v_b-1} \prod_{c=1}^C
(r_c-1)r_c^{w_c-1}\\
&=&\frac{\omega(N\mathcal{O}_K)}{2}\phi(N) \quad\textrm{where $\phi$
is
the Euler function for integers}\\
&=&\#~\{\left(\begin{smallmatrix}t&0\\0&t\end{smallmatrix}\right)~:~t\in(\mathbb{Z}/N\mathbb{Z})^*\}
/\{\pm\left(\begin{smallmatrix}1&0\\0&1\end{smallmatrix}\right)\}.
\end{eqnarray*}
This concludes the proposition.
\end{proof}

\begin{remark}
Lemma \ref{jfix} and Proposition \ref{Gal(K_N/H_O)} have been given
in \cite[Lemma 5.2 and Proposition 5.3]{K-S2} without much
explanation. For completeness we present their proof in detail.
\end{remark}

\section{Ring class invariants}\label{section4}

We shall make use of Theorem \ref{generator} to construct primitive
generators of ring class fields as the singular values of certain
$\Delta$-quotients.

\begin{lemma}\label{StoD}
For a positive integer $N$, we have the relation
\begin{equation*}
\prod_{t=1}^{N-1}g_{(0,t/N)}(\tau)^{12}=
N^{12}\Delta(N\tau)/\Delta(\tau),
\end{equation*}
where the left hand side is regarded as $1$ when $N=1$.
\end{lemma}
\begin{proof}
For $N\geq2$ we find that
\begin{eqnarray*}
&&\prod_{t=1}^{N-1}g_{(0,t/N)}(\tau)^{12}\\
&=& \prod_{t=1}^{N-1}\bigg(-q^{1/12}
\zeta_{2N}^{-t}(1-\zeta_N^t)\prod_{n=1}^\infty(1-q^n\zeta_N^t)
(1-q^n\zeta_N^{-t})\bigg)^{12}\quad\textrm{by the definition
(\ref{FourierSiegel})}\\
&=&q^{N-1}N^{12}\prod_{n=1}^\infty((1-q^{Nn})/(1-q^n))^{24}
\quad\textrm{by the identity $\prod_{t=1}^{N-1}(1-\zeta_N^tX)=(1-X^N)/(1-X)$}\\
&=&N^{12}\Delta(N\tau)/\Delta(\tau)\quad\textrm{by the definition
(\ref{Delta})}.
\end{eqnarray*}
\end{proof}

We are ready to prove our first main theorem.

\begin{theorem}\label{main}
Let $K$ be an imaginary quadratic field and $\theta$ be as in
\textup{(\ref{theta})}. Let $\mathcal{O}$ be the order of conductor
$N=\prod_{k=1}^n p_k^{e_k}$ \textup{(}$\geq2$\textup{)} in $K$. Set
\begin{equation*}
N_S=\left\{\begin{array}{ll}
\prod_{k\in S}p_k & \textrm{if $S$ is a nomepty subset of $\{1,2,\cdots,n\}$}\vspace{0.1cm}\\
\displaystyle1 & \textrm{if $S=\emptyset$.}\end{array}\right.
\end{equation*}
If $\mathfrak{f}=N\mathcal{O}_K$ satisfies the assumption
\textup{(\ref{hypothesis})} in Lemma \textup{\ref{character}}, then
the singular value
\begin{equation}\label{ringclassinvariant}
\left\{\begin{array}{ll}
p_1^{12}\Delta(p_1^{e_1}\theta)/\Delta(p_1^{e_1-1}\theta) &
\textrm{if
$n=1$}\vspace{0.1cm}\\
\prod_{S\subseteq\{1,2,\cdots,n\}}
\Delta((N/N_S)\theta)^{(-1)^{\#S}} & \textrm{if $n\geq2$}
\end{array}\right.
\end{equation}
generates $H_\mathcal{O}$ over $K$ as a real algebraic integer.
\end{theorem}
\begin{proof}
If $\mathfrak{f}=N\mathcal{O}_K$, then
$g_\mathfrak{f}(C_0)=g_{(0,1/N)}(\theta)^{12N}$ by the definition.
We get that
\begin{eqnarray}
&&\left\{\begin{array}{ll}\mathbf{N}_{K_\mathfrak{f}/H_\mathcal{O}}(g_\mathfrak{f}(C_0))
&\textrm{if $N=2$}\vspace{0.1cm}\\
\mathbf{N}_{K_\mathfrak{f}/H_\mathcal{O}}(g_\mathfrak{f}(C_0))^2
&\textrm{if $N\geq3$}\end{array}\right.\nonumber\\
&=&\prod_{\begin{smallmatrix}1\leq t\leq
N-1\\\gcd(t,N)=1\end{smallmatrix}}
(g_{(0,1/N)}(\theta)^{12N})^{\left(\begin{smallmatrix}t&0\\0&t\end{smallmatrix}\right)}\quad\textrm{by
Proposition \ref{Gal(K_N/H_O)}}\nonumber\\
&=&\prod_{\begin{smallmatrix}1\leq t\leq
N-1\\\gcd(t,N)=1\end{smallmatrix}}
(g_{(0,1/N)}(\tau)^{12N})^{\left(\begin{smallmatrix}t&0\\0&t\end{smallmatrix}\right)}(\theta)\quad\textrm{by
Proposition \ref{reciprocity}}\nonumber\\
&=&\prod_{\begin{smallmatrix}1\leq t\leq
N-1\\\gcd(t,N)=1\end{smallmatrix}}g_{(0,t/N)}(\theta)^{12N}
\quad\textrm{by Proposition
\ref{TransformSiegel}(ii)}\nonumber\\
&=&\prod_{S\subseteq\{1,2,\cdots,n\}} \bigg(\prod_{1\leq t\leq
N-1,N_S|t}
g_{(0,t/N)}(\theta)^{12}\bigg)^{N(-1)^{\#S}}\quad\textrm{by
inclusion-exclusion principle}\nonumber\\
&=&\prod_{S\subseteq\{1,2,\cdots,n\}}
\bigg(\prod_{w=1}^{(N/N_S)-1}g_{(0,N_Sw/N)}(\theta)^{12}\bigg)^{N(-1)^{\#S}}
\quad\textrm{by setting $t=N_Sw$}\nonumber\\
&=&\prod_{S\subseteq\{1,2,\cdots,n\}}
((N/N_S)^{12}\Delta((N/N_S)\theta)/\Delta(\theta))^{N(-1)^{\#S}}\quad\textrm{by
Lemma \ref{StoD}},\label{step1}
\end{eqnarray}
which is a generator of $H_\mathcal{O}$ over $K$ by Theorem
\ref{generator} and Remark \ref{remarkgenerator}. On the other hand,
the value
$\mathbf{N}_{K_\mathfrak{f}/H_\mathcal{O}}(g_\mathfrak{f}(C_0))$ is
an algebraic integer by Propositions \ref{TransformSiegel}(iii) and
\ref{jlemma}(i). Furthermore, each factor
$\Delta((N/N_S)\theta)/\Delta(\theta)$ appeared in (\ref{step1})
belongs to the ring class field of the order of conductor $N/N_S$ in
$K$ as a real algebraic number by Proposition \ref{jlemma}(ii).
Therefore the value in (\ref{step1}) without $N^\textrm{th}$ power
generates $H_\mathcal{O}$ over $K$ as an algebraic integer. We
further observe that
\begin{eqnarray*}
&&\prod_{S\subseteq\{1,2,\cdots,n\}}
((N/N_S)^{12}\Delta((N/N_S)\theta)/\Delta(\theta))^{(-1)^{\#S}}\\
&=&(N^{12}/\Delta(\theta))^{\sum_{S\subseteq\{1,2,\cdots,n\}}
(-1)^{\#S}}
\prod_{S\subseteq\{1,2,\cdots,n\}}N_S^{-12(-1)^{\#S}}\prod_{S\subseteq\{1,2,\cdots,n\}}
\Delta((N/N_S)\theta)^{(-1)^{\#S}}\\
&=&\left\{\begin{array}{ll}(p_1^{12e_1}/\Delta(\theta))^{1-1}
p_1^{12}\Delta(p_1^{e_1}\theta)
\Delta(p_1^{e_1-1}\theta)^{-1} & \textrm{if $n=1$}\vspace{0.1cm}\\
(N^{12}/\Delta(\theta))^{\sum_{k=0}^n\left(\begin{smallmatrix}n\\k\end{smallmatrix}\right)(-1)^k}
\prod_{k=1}^n
p_k^{-12\sum_{\ell=1}^{n}\left(\begin{smallmatrix}n-1\\\ell-1\end{smallmatrix}\right)(-1)^\ell}
&\vspace{0.1cm}\\
\times\prod_{S\subseteq\{1,2,\cdots,n\}}\Delta((N/N_S)\theta)^{(-1)^{\#S}}
& \textrm{if $n\geq2$}\end{array}\right.\\
&=&\left\{\begin{array}{ll}
p_1^{12}\Delta(p_1^{e_1}\theta)/\Delta(p_1^{e_1-1}\theta) &
\textrm{if
$n=1$}\vspace{0.1cm}\\
\prod_{S\subseteq\{1,2,\cdots,n\}}\Delta((N/N_S)\theta)^{(-1)^{\#S}}
& \textrm{if $n\geq2$.}
\end{array}\right.
\end{eqnarray*}
This completes the proof.
\end{proof}

\begin{remark}\label{mainremark}
Let $\mathcal{O}$ be the order of conductor $N=\prod_{k=1}^n
p_k^{e_k}$ ($\geq2$) in an imaginary quadratic field $K$. Denote by
$r_k$ the ramification index of $p_k$ in $K/\mathbb{Q}$ for each
$k=1,\cdots,n$. Assume first that
\begin{equation}\label{newassumption1}
\textrm{each $p_k$ is an odd prime which is inert or ramified in
$K/\mathbb{Q}$.}
\end{equation}
Since $N\mathcal{O}_K=\prod_{k=1}^n
\mathfrak{p}_k^{r_ke_k}~\textrm{with}~\mathbf{N}_{K/\mathbb{Q}}(\mathfrak{p}_k)=p_k^{2/r_k}$,
we have
\begin{equation*}
\widehat{\varepsilon}_k=\left\{\begin{array}{ll}
(1/\omega_K)(p_1^{2/r_1}-1)p_1^{2e_1-2/r_1} & \textrm{if
$n=1$}\vspace{0.1cm}\\(p_k^{2/r_k}-1)p_k^{2e_k-2/r_k} & \textrm{if
$n\geq2$}
\end{array}\right.
,\quad \varepsilon_k=\frac{\prod_{\ell=1}^n(p_\ell^{2/r_\ell}-1)
p_\ell^{2e_\ell-2/r_\ell}} {(p_k^{2/r_k}-1)p_k^{2e_k-2/r_k}}
\quad(k=1,\cdots,n),
\end{equation*}
and $\#~\mathrm{Cl}(K_{(N)}/H_\mathcal{O})=(1/2)\prod_{k=1}^n
(p_k-1)p_k^{e_k-1}$ by Proposition \ref{Gal(K_N/H_O)}. Assume
further that
\begin{equation}\label{newassumption2}
e_k+1>2/r_k~(k=1,\cdots,n)~\textrm{and}~ \left\{\begin{array}{ll}
\gcd(p_1,\omega_K)=1
& \textrm{if $n=1$}\vspace{0.1cm}\\
\gcd(\prod_{k=1}^np_k,\prod_{k=1}^n(p_k^{2/r_k}-1))=1 & \textrm{if
$n\geq2$.}
\end{array}\right.
\end{equation}
Then, since
\begin{equation*}
p_k\nmid\varepsilon_k~\textrm{and}~\mathrm{ord}_{p_k}(\widehat{\varepsilon}_k)
=2e_k-2/r_k>
\mathrm{ord}_{p_k}(\#~\mathrm{Cl}(K_{(N)}/H_\mathcal{O}))=e_k-1~(k=1,\cdots,n),
\end{equation*}
we can take $\nu_k=p_k$ as for the assumption (\ref{hypothesis}) in
Lemma \ref{character}. Therefore one can apply Theorem \ref{main}
under the assumptions (\ref{newassumption1}) and
(\ref{newassumption2}).
\end{remark}

\begin{example}\label{exer1}
If $K=\mathbb{Q}(\sqrt{-7})$, then $\theta=(-1+\sqrt{-7})/2$ and
$h_K=1$ (\cite[Theorem 12.34]{Cox}), in other words, $H=K$. Let
$\mathcal{O}$ be the order of conductor $N=7$ in $K$. We get by
Propositions \ref{reciprocity} and \ref{Gal(K_N/H_O)} that
\begin{eqnarray*}
\mathrm{Gal}(H_\mathcal{O}/K)&\simeq&(W_{7,\theta}
/\{\pm\left(\begin{smallmatrix}1&0\\0&1\end{smallmatrix}\right)\}
)/(\{\left(\begin{smallmatrix}t&0\\0&t\end{smallmatrix}\right):t\in(\mathbb{Z}/7\mathbb{Z})^*\}/
\{\pm\left(\begin{smallmatrix}1&0\\0&1\end{smallmatrix}\right)\})\\
&=&\{ \left(\begin{smallmatrix}1&0\\0&1\end{smallmatrix}\right)\cdot
\left(\begin{smallmatrix}1&0\\0&1\end{smallmatrix}\right),
\left(\begin{smallmatrix}1&0\\0&2\end{smallmatrix}\right)\cdot
\left(\begin{smallmatrix}-1&-2\\4&7\end{smallmatrix}\right),
\left(\begin{smallmatrix}1&0\\0&2\end{smallmatrix}\right)\cdot
\left(\begin{smallmatrix}7&12\\-10&-17\end{smallmatrix}\right),
\left(\begin{smallmatrix}1&0\\0&4\end{smallmatrix}\right)\cdot
\left(\begin{smallmatrix}-13&-16\\9&11\end{smallmatrix}\right),\\
&&\phantom{\{}
\left(\begin{smallmatrix}1&0\\0&1\end{smallmatrix}\right)\cdot
\left(\begin{smallmatrix}-5&-16\\1&3\end{smallmatrix}\right),
\left(\begin{smallmatrix}1&0\\0&1\end{smallmatrix}\right)\cdot
\left(\begin{smallmatrix}-3&-16\\1&5\end{smallmatrix}\right),
\left(\begin{smallmatrix}1&0\\0&4\end{smallmatrix}\right)\cdot
\left(\begin{smallmatrix}-16&-9\\9&5\end{smallmatrix}\right)\}.
\end{eqnarray*}
Note that we expressed elements of $\mathrm{Gal}(H_\mathcal{O}/K)$
in the form of
\begin{equation*}
\begin{pmatrix}1&0\\0&d\end{pmatrix}~\textrm{for
some $d\in(\mathbb{Z}/7\mathbb{Z})^*$}\cdot\textrm{an element of
$\mathrm{SL}_2(\mathbb{Z})$}.
\end{equation*}
On the other hand, since $7$ is ramified in $K/\mathbb{Q}$ and
$\omega_K=2$, the assumptions (\ref{newassumption1}) and
(\ref{newassumption2}) in Remark \ref{mainremark} (or, the
assumption (\ref{newassumption3}) in Remark \ref{newremark}) are
satisfied. Hence the singular value
$7^{12}\Delta(7\theta)/\Delta(\theta)$ generates $H_\mathcal{O}$
over $K$ by Theorem \ref{main} (or, Theorem \ref{main2}).
Furthermore, since the function $\Delta(7\tau)/\Delta(\tau)$ belongs
to $\mathcal{F}_7$ and has rational Fourier coefficients by
Proposition \ref{functionj}(iii), we obtain its minimal polynomial
by Propositions \ref{Gal(F_N/F_1)} and \ref{reciprocity} as
\begin{eqnarray*}
&&\mathrm{min}(7^{12}\Delta(7\theta)/\Delta(\theta),K)\\&=&
(X-(7^{12}\Delta(7\tau)/\Delta(\tau))\circ\left(\begin{smallmatrix}1&0\\0&1\end{smallmatrix}\right)(\theta))
(X-(7^{12}\Delta(7\tau)/\Delta(\tau))\circ\left(\begin{smallmatrix}-1&-2\\4&7\end{smallmatrix}\right)(\theta))\\
&&(X-(7^{12}\Delta(7\tau)/\Delta(\tau))\circ\left(\begin{smallmatrix}7&12\\-10&-17\end{smallmatrix}\right)(\theta))
(X-(7^{12}\Delta(7\tau)/\Delta(\tau))\circ\left(\begin{smallmatrix}-13&-16\\9&11\end{smallmatrix}\right)(\theta))\\
&&(X-(7^{12}\Delta(7\tau)/\Delta(\tau))\circ\left(\begin{smallmatrix}-5&-16\\1&3\end{smallmatrix}\right)(\theta))
(X-(7^{12}\Delta(7\tau)/\Delta(\tau))\circ\left(\begin{smallmatrix}-3&-16\\1&5\end{smallmatrix}\right)(\theta))\\
&&(X-(7^{12}\Delta(7\tau)/\Delta(\tau))\circ\left(\begin{smallmatrix}-16&-9\\9&5\end{smallmatrix}\right)(\theta))\\
&=&X^7+234857X^6+24694815621X^5+295908620105035X^4+943957383096939785X^3\\
&&+356807315211847521X^2+38973886319454982X-117649.
\end{eqnarray*}
On the other hand, if we compare its coefficients with those of the
minimal polynomial of the classical invariant $j(7\theta)$, we see
in a similar fashion that the latter are much bigger than the former
as follows:
\begin{eqnarray*}
&&\mathrm{min}(j(7\theta),K)\\
&=&X^7
+18561099067532582351348250X^6+54379116263846797396254926859375X^5\\
&&+344514398594838596665876837347342843995647646484375X^4\\
&&+1009848457088842748174122781381460720529620832094970703125X^3\\
&&+1480797351289795967859364968037513969226011238564633514404296875
X^2\\
&&-3972653601649066484326573605251406741304015473521796878814697265625X\\
&&+4791576562341747034548276661270093305105027267573103845119476318359375.
\end{eqnarray*}
\end{example}

\begin{example}\label{exer2}
Let $K=\mathbb{Q}(\sqrt{-5})$ and $\mathcal{O}$ be the order of
conductor $N=6$ ($=2\cdot3$) in $K$. One can readily check that
$N\mathcal{O}_K$ satisfies neither the assumption (\ref{hypothesis})
in Lemma \ref{character} nor the assumption (\ref{newassumption3})
in Remark \ref{newremark}. Even in this case, however, we will see
that our method is still valid. Therefore, it is worth of studying
how much further one can release from the assumption
(\ref{hypothesis}) in Lemma \ref{character} (or, the assumption
(\ref{newassumption3}) in Remark \ref{newremark}).
\par
Observe that $h_K=2$ (\cite[p.29]{Cox}) and $[H_\mathcal{O}:K]=8$ by
Lemma \ref{degree}(i). Since $h_K=2$, there are two reduced positive
definite binary quadratic forms of discriminant $d_K=-20$, namely
\begin{equation*}
Q_1=X^2+5Y^2~\textrm{and}~Q_2=2X^2+2XY+3Y^2.
\end{equation*}
We associate to each $Q_k$ ($k=1,2$) a matrix in
$\mathrm{GL}_2(\mathbb{Z}/N\mathbb{Z})$ and a CM-point as follows:
\begin{eqnarray*}\left\{\begin{array}{ll}
\beta_1=\left(\begin{smallmatrix}1&0\\0&1\end{smallmatrix}\right),~
\theta_1=\sqrt{-5} & \textrm{for $Q_1$}\vspace{0.1cm}\\
\beta_2=\left(\begin{smallmatrix}1&5\\3&2\end{smallmatrix}\right),~
\theta_2=(-1+\sqrt{-5})/2 & \textrm{for $Q_2$}.
\end{array}\right.
\end{eqnarray*}
Then we see from \cite[$\S$6]{Stevenhagen} that
\begin{eqnarray*}
\mathrm{Gal}(H/K)=\{(h(\theta)\mapsto
h^{\beta_k}(\theta_k))\big|_H~:~k=1,2\},
\end{eqnarray*}
where $h\in\mathcal{F}_N$ is defined and finite at
$\theta=\sqrt{-5}$. Furthermore, it follows from Propositions
\ref{reciprocity} and \ref{Gal(K_N/H_O)} that
\begin{equation*}
\mathrm{Gal}(H_\mathcal{O}/H)\simeq\{
\alpha_1=\left(\begin{smallmatrix}1&0\\0&1\end{smallmatrix}\right),
\alpha_2=\left(\begin{smallmatrix}0&1\\1&0\end{smallmatrix}\right),
\alpha_3=\left(\begin{smallmatrix}2&3\\3&2\end{smallmatrix}\right),
\alpha_4=\left(\begin{smallmatrix}3&2\\2&3\end{smallmatrix}\right)\}.
\end{equation*}
Hence we achieve that
\begin{equation*}
\mathrm{Gal}(H_\mathcal{O}/K)=\{(h(\theta)\mapsto
h^{\alpha_\ell\beta_k}(\theta_k))|_{H_\mathcal{O}}~:~\ell=1,\cdots,4,~k=1,2\},
\end{equation*}
where $h\in\mathcal{F}_N$ is defined and finite at $\theta$. The
conjugates of
$\Delta(6\theta)\Delta(\theta)/\Delta(2\theta)\Delta(3\theta)$
estimated according to Theorem \ref{main} are
\begin{equation*}
x_{\ell,k}=(\Delta(6\tau)\Delta(\tau)/
\Delta(2\tau)\Delta(3\tau))^{\alpha_\ell\beta_k}(\theta_k)\quad(\ell=1,\cdots,4,~k=1,2)
\end{equation*}
possibly with some multiplicity. And, since the function
$\Delta(6\tau)\Delta(\tau)/\Delta(2\tau)\Delta(3\tau)\in\mathcal{F}_N$
has rational Fourier coefficients, the action of each
$\alpha_\ell\beta_k$ on it can be performed as in the previous
example. Thus the minimal polynomial of
$\Delta(6\theta)\Delta(\theta)/\Delta(2\theta)\Delta(3\theta)$
becomes a divisor of
\begin{eqnarray*}
\textstyle{\prod_{\ell=1,\cdots,4,k=1,2}}(X-x_{\ell,k})
&=&X^8-1304008X^7+16670918428X^6+30056736254344X^5\\
&&+23344024601638470X^4
+7327603919934344X^3\\
&&+1949665164230428X^2-1597207512008X+1.
\end{eqnarray*}
This polynomial is, however, irreducible and hence the singular
value $\Delta(6\theta)\Delta(\theta)/\Delta(2\theta)\Delta(3\theta)$
should be a primitive generator of $H_\mathcal{O}$ over $K$.
\end{example}

\section{Another approach}\label{section5}

We shall develop a different lemma which substitutes for Lemma
\ref{character}, from which we are able to find more $N$'s in
Theorem \ref{main}.
\par
Throughout this section, we let $K$ be an imaginary quadratic field
other than $\mathbb{Q}(\sqrt{-1})$, $\mathbb{Q}(\sqrt{-3})$, and
$\theta$ be as in (\ref{theta}). Let $\mathcal{O}$ be the order of
conductor $N$ ($\geq2$) in $K$ with
\begin{equation*}
\mathfrak{f}=N\mathcal{O}_K=\prod_{k=1}^n\mathfrak{p}_k^{e_k}.
\end{equation*}
We use the same notations $\pi_\mathfrak{f}$, $\iota$, $\iota_k$,
$v_k$, $\widetilde{\Phi}_\mathfrak{f}$ as in $\S$\ref{section2}.
And, by $\mathrm{Cl}(H_\mathcal{O}/K )$ we mean the quotient group
of $\mathrm{Cl}(\mathfrak{f})$ corresponding to
$\mathrm{Gal}(H_\mathcal{O}/K)$ via the Artin map, that is
\begin{equation}\label{Cl(H_O/K)}
\mathrm{Cl}(H_\mathcal{O}/K)=\mathrm{Cl}(\mathfrak{f})/\mathrm{Cl}(K_\mathfrak{f}/H_\mathcal{O}).
\end{equation}
We further let $\mathrm{Cl}(H_\mathcal{O}/H)$ stand for the subgroup
of $\mathrm{Cl}(H_\mathcal{O}/K)$ corresponding to
$\mathrm{Gal}(H_\mathcal{O}/H)$.
\par
Setting
\begin{equation}\label{Psi}
\check\Psi_\mathfrak{f}=(\mathrm{Cl}(\mathfrak{f})\rightarrow\mathrm{Cl}(H_\mathcal{O}/K))\circ\widetilde{\Phi}_\mathfrak{f}
~:~\pi_\mathfrak{f}(\mathcal{O}_K)^*\longrightarrow\mathrm{Cl}(H_\mathcal{O}/K),
\end{equation}
we obtain from the exact sequence (\ref{exact}) and Galois theory
another exact sequence
\begin{equation}\label{exact2}
1\longrightarrow
\pi_\mathfrak{f}(\mathcal{O}_K)^*/\mathrm{Ker}(\check{\Psi}_\mathfrak{f})
\longrightarrow
\mathrm{Cl}(H_\mathcal{O}/K)\longrightarrow\mathrm{Cl}(\mathcal{O}_K)
\longrightarrow1
\end{equation}
with
\begin{equation}\label{imageCl}
\check{\Psi}_\mathfrak{f}(\pi_\mathfrak{f}(\mathcal{O}_K)^*)=\mathrm{Cl}(H_\mathcal{O}/H).
\end{equation}
We know by the fact $\omega_K=2$ and Lemma \ref{degree} that
\begin{eqnarray*}
\#\pi_\mathfrak{f}(\mathcal{O}_K)^*/\pi_\mathfrak{f}(\mathbb{Z})^*=
\varphi(\mathfrak{f})/\phi(N)
~\textrm{and}~[H_\mathcal{O}:H]=\varphi(\mathfrak{f})/\phi(N).
\end{eqnarray*}
On the other hand, since
$\mathrm{Cl}(K_\mathfrak{f}/H_\mathcal{O})=P_{K,\mathbb{Z}}(\mathfrak{f})/P_{K,1}(\mathfrak{f})$
by the definition of $H_\mathcal{O}$, we get
$\pi_\mathfrak{f}(\mathbb{Z})^*\subseteq\mathrm{Ker}(\check{\Psi}_\mathfrak{f})$;
hence we achieve
\begin{equation}\label{kernelPsi}
\mathrm{Ker}(\check{\Psi}_\mathfrak{f})=\pi_\mathfrak{f}(\mathbb{Z})^*.
\end{equation}

\begin{lemma}\label{quotientextend}
Let $G$ be a finite abelian group and $H$ be a subgroup of $G$.
There is a canonical isomorphism between character groups
\begin{equation}\label{charactermap}
\begin{array}{ccc}
\{\textrm{characters of}~G~\textrm{which are trivial
on}~H\}&\longrightarrow&\{\textrm{characters
of}~G/H\}\vspace{0.1cm}\\
\chi&\mapsto&(gH\mapsto\chi(g)~:~g\in G).
\end{array}
\end{equation}
\end{lemma}
\begin{proof}
One can readily check that the map in (\ref{charactermap}) is a
well-defined injection. For surjectivity, let $\psi$ be a character
of $G/H$. Then the character
\begin{equation*}
\chi=\psi\circ(G\rightarrow G/H)
\end{equation*}
of $G$ maps to $\psi$ via the map in (\ref{charactermap}), which
claims the surjectivity.
\end{proof}

Thus we have a canonical isomorphism
\begin{eqnarray}\label{canonical}
\{\textrm{characters of}~\mathrm{Cl}(\mathfrak{f})~\textrm{which are
trivial
on}~\mathrm{Cl}(K_\mathfrak{f}/H_\mathcal{O})\}\longrightarrow\{\textrm{characters
of}~\mathrm{Cl}(H_\mathcal{O}/K)\}
\end{eqnarray}
by Lemma \ref{quotientextend} and definition (\ref{Cl(H_O/K)}). For
any character $\psi$ of $\mathrm{Cl}(H_\mathcal{O}/K)$ we define
\begin{equation*}
\check{\psi}=\psi\circ\check{\Psi}_\mathfrak{f}~\textrm{and}~
\check{\psi}_k=\check{\psi}\circ\iota\circ\iota_k~(k=1,\cdots,n).
\end{equation*}
If $\chi$ maps to $\psi$ via the map in (\ref{canonical}), then we
derive
\begin{eqnarray*}
\widetilde{\chi}&=&\chi\circ\widetilde{\Phi}_\mathfrak{f}=
\psi\circ(\mathrm{Cl}(\mathfrak{f})\rightarrow\mathrm{Cl}(H_\mathcal{O}/K))
\circ\widetilde{\Phi}_\mathfrak{f}\quad\textrm{by the proof of Lemma
\ref{quotientextend}}\\
&=&\psi\circ\check{\Psi}_\mathfrak{f}=\check{\psi}\quad\textrm{by
definition (\ref{Psi})},
\end{eqnarray*}
from which it follows that
\begin{equation*}
\widetilde{\chi}_k=\check{\psi}_k\quad(k=1,\cdots,n).
\end{equation*}

\begin{lemma}\label{character2}
Let
\begin{eqnarray*}
U&=&\{\textrm{characters
of}~\mathrm{Cl}(H_\mathcal{O}/K)~\textrm{which are trivial on}~{\mathrm{Cl}(H_\mathcal{O}/H)}\},\\
V&=&\{\textrm{characters
of}~\mathrm{Cl}(H_\mathcal{O}/H)\},\\
W&=&\{\textrm{characters
of}~\mathrm{Cl}(H_\mathcal{O}/K)\},\\
G_k&=&\widehat{v}_k
\circ\iota^{-1}(\pi_\mathfrak{f}(\mathbb{Z})^*)\quad(k=1,\cdots,n),
\end{eqnarray*}
where
\begin{equation*}
\widehat{v}_k~:~
\prod_{\ell=1}^n\pi_{\mathfrak{p}_\ell^{e_\ell}}(\mathcal{O}_K)^*
\longrightarrow \pi_{\mathfrak{p}_1^{e_1}}(\mathcal{O}_K)^*\times
\cdots\times\pi_{\mathfrak{p}_{k-1}^{e_{k-1}}}(\mathcal{O}_K)^*\times
\pi_{\mathfrak{p}_{k+1}^{e_{k+1}}}(\mathcal{O}_K)^*\times
\cdots\times\pi_{\mathfrak{p}_{n}^{e_{n}}}(\mathcal{O}_K)^*
\end{equation*}
is the natural projection which deletes the $k$-th component.
 For each character $\psi\in V$,
fix a character $\psi'\in W$ which extends $\psi$ \textup{(}by Lemma
\textup{\ref{characterextend})}.
\begin{itemize}
\item[(i)]
There is a bijective map
\begin{eqnarray*}
U\times V&\longrightarrow&W\\
(\chi,\psi)&\mapsto&\chi\cdot\psi'.\nonumber
\end{eqnarray*}
\item[(ii)] We have the inequality
\begin{equation*}
\#\{\xi\in W~:~\check{\xi}_k=1\} \leq
h_K\frac{\#\pi_\mathfrak{f}(\mathcal{O}_K)^*}
{\#\pi_{\mathfrak{p}_k^{e_k}}(\mathcal{O}_K)^*\cdot\#G_k}
\quad(k=1,\cdots,n).
\end{equation*}
\end{itemize}
\end{lemma}
\begin{proof}
(i) We see from Lemma \ref{quotientextend} that both $U\times V$ and
$W$ have the same size. Hence it suffices to show that the above map
is injective, which is straightforward.\\
(ii) Without loss of generality it suffices to show that there is an
injective map
\begin{equation*}
S=\{\xi\in W~:~\check{\xi}_n=1\}\longrightarrow
U\times\{\textrm{characters
of}~\prod_{k=1}^{n-1}\pi_{\mathfrak{p}_k^{e_k}}(\mathcal{O}_K)^*/G_n\},
\end{equation*}
because $\#U=h_K$ by Lemma \ref{quotientextend} and
$\#\prod_{k=1}^{n-1}\pi_{\mathfrak{p}_k^{e_k}}(\mathcal{O}_K)^*/G_n=
\#\pi_\mathfrak{f}(\mathcal{O}_K)^*/(\#\pi_{\mathfrak{p}_n^{e_n}}(\mathcal{O}_K)^*\cdot\#
G_n)$.
\par
Let $\xi\in S$. As an element of $W$, $\xi$ is of the form
$\chi\cdot\psi'$ for some $\chi\in U$ and $\psi\in V$ by (i). And,
by (\ref{imageCl}) and the fact $\chi\in U$ we get
$\check{\chi}=\chi\circ\check{\Psi}_\mathfrak{f}=1$, from which it
follows that
\begin{equation}\label{n-th}
1=\check{\xi}_n=(\check{\chi}\cdot\check{\psi'})_n=\check{\psi'}_n.
\end{equation}
We further deduce by (\ref{imageCl}) that
\begin{equation}\label{equal}
\check{\psi'}=\psi'|_{\check{\Psi}_\mathfrak{f}(\pi_\mathfrak{f}(\mathcal{O}_K)^*)}\circ\check{\Psi}_\mathfrak{f}=
\psi'|_{\mathrm{Cl}(H_\mathcal{O}/H)}\circ\check{\Psi}_\mathfrak{f}=\psi\circ\check{\Psi}_\mathfrak{f}.
\end{equation}
\par
On the other hand, if $\beta$ is a character of
$\prod_{k=1}^{n}\pi_{\mathfrak{p}_k^{e_k}}(\mathcal{O}_K)^*$ defined
by
\begin{equation}\label{beta}
\beta=\psi\circ\check{\Psi}_\mathfrak{f}\circ\iota,
\end{equation}
then we derive that
\begin{eqnarray*}
\beta\circ\iota_n(\pi_{\mathfrak{p}_n^{e_n}}(\mathcal{O}_K)^*)
&=&\psi\circ\check{\Psi}_\mathfrak{f}\circ\iota\circ\iota_n(\pi_{\mathfrak{p}_n^{e_n}}(\mathcal{O}_K)^*)\\
&=&\check{\psi}'\circ\iota\circ\iota_n(\pi_{\mathfrak{p}_n^{e_n}}(\mathcal{O}_K)^*)\quad\textrm{by (\ref{equal})}\\
&=&\check{\psi'}_n(\pi_{\mathfrak{p}_n^{e_n}}(\mathcal{O}_K)^*)=1\quad\textrm{by
(\ref{n-th})},
\end{eqnarray*}
which implies
\begin{equation}\label{kernelinclusion1}
\iota_n(\pi_{\mathfrak{p}_n^{e_n}}(\mathcal{O}_K)^*)\subseteq\mathrm{Ker}(\beta).
\end{equation}
Furthermore, we have
$\beta\circ\iota^{-1}(\pi_\mathfrak{f}(\mathbb{Z})^*)=\psi
\circ\check{\Psi}_\mathfrak{f}(\pi_\mathfrak{f}(\mathbb{Z})^*)=1$ by
(\ref{kernelPsi}), which claims
\begin{equation}\label{kernelinclusion2}
\iota^{-1}(\pi_\mathfrak{f}(\mathbb{Z})^*)\subseteq\mathrm{Ker}(\beta).
\end{equation}
Hence $\beta$ can be written as
\begin{equation}\label{betagamma}
\beta=\gamma\circ(
\prod_{k=1}^{n}\pi_{\mathfrak{p}_k^{e_k}}(\mathcal{O}_K)^*\rightarrow
\prod_{k=1}^{n}\pi_{\mathfrak{p}_k^{e_k}}(\mathcal{O}_K)^*/ \langle
\iota_n(\pi_{\mathfrak{p}_n^{e_n}}(\mathcal{O}_K)^*),\iota^{-1}(\pi_\mathfrak{f}(\mathbb{Z})^*)
\rangle )
\end{equation}
for a unique character $\gamma$ of
$\prod_{k=1}^{n}\pi_{\mathfrak{p}_k^{e_k}}(\mathcal{O}_K)^*/\langle
\iota_n(\pi_{\mathfrak{p}_n^{e_n}}(\mathcal{O}_K)^*),\iota^{-1}(\pi_\mathfrak{f}(\mathbb{Z})^*)
\rangle$ by Lemma \ref{quotientextend}, (\ref{kernelinclusion1}) and
(\ref{kernelinclusion2}).
\par
Now, we define a map
\begin{eqnarray*}
\kappa~:~S&\longrightarrow&U\times \{\textrm{characters
of}~\prod_{k=1}^{n-1}\pi_{\mathfrak{p}_k^{e_k}}(\mathcal{O}_K)^*/G_n\}\\
\xi&\mapsto&(\chi,\gamma\circ\widehat{\iota}_n),
\end{eqnarray*}
where
\begin{equation*}
\widehat{\iota}_n~:~\prod_{k=1}^{n-1}\pi_{\mathfrak{p}_k^{e_k}}(\mathcal{O}_K)^*/G_n
\longrightarrow\prod_{k=1}^{n}\pi_{\mathfrak{p}_k^{e_k}}(\mathcal{O}_K)^*/
\langle
\iota_n(\pi_{\mathfrak{p}_n^{e_n}}(\mathcal{O}_K)^*),\iota^{-1}(\pi_\mathfrak{f}(\mathbb{Z})^*)
\rangle
\end{equation*}
is definitely a surjection by the definition of $G_n$. To prove the
injectivity of the map $\kappa$, assume that
$\kappa(\xi_1)=\kappa(\xi_2)$ for some $\xi_1,\xi_2\in S$. Then, by
(i) there are unique $\chi_1,\chi_2\in U$ and $\psi_1,\psi_2\in V$
such that $\xi_1=\chi_1\cdot\psi_1'$ and $\xi_2=\chi_2\cdot\psi_2'$.
And, by the definition of $\kappa$ we get $\chi_1=\chi_2$. Let
$\psi_\ell$ ($\ell=1,2$) induce $\beta_\ell$ and $\gamma_\ell$ in
the above paragraph (which explains $\beta$ and $\gamma$ constructed
from $\psi$). Then, since $\widehat{\iota}_n$ is surjective, we
obtain $\gamma_1=\gamma_2$ from the fact
$\gamma_1\circ\widehat{\iota}_n=\gamma_2\circ\widehat{\iota}_n$, and
so we have $\beta_1=\beta_2$ by (\ref{betagamma}). It then follows
from the definition (\ref{beta}), the fact $\psi_1,\psi_2\in V$ and
(\ref{imageCl}) that $\psi_1=\psi_2$, which concludes the
injectivity of $\kappa$. This completes the proof.
\end{proof}

\begin{lemma}\label{character3}
Let $F$ be a field such that $K\subseteq F\subsetneq H_\mathcal{O}$.
If
\begin{equation}\label{assumption}
2\#\pi_\mathfrak{f}(\mathbb{Z})^*\sum_{k=1}^n\frac{1}{\#\pi_{\mathfrak{p}_k^{e_k}}(\mathcal{O}_K)^*\cdot\#G_k}<1,
\end{equation}
then there is a character $\chi$ of $\mathrm{Cl}(\mathfrak{f})$ such
that
\begin{equation}\label{condition}
\chi|_{\mathrm{Cl}(K_\mathfrak{f}/H_\mathcal{O})}=1,~
\chi|_{\mathrm{Cl}(K_\mathfrak{f}/F)}\neq1~\textrm{and}~
\mathfrak{p}_k|\mathfrak{f}_\chi\quad(k=1,\cdots,n).
\end{equation}
\end{lemma}
\begin{proof}
We first derive that
\begin{eqnarray*}
&&\#\{\textrm{characters}~\chi~\textrm{of}~\mathrm{Cl}(\mathfrak{f})~:~
\chi|_{\mathrm{Cl}(K_\mathfrak{f}/H_\mathcal{O})}=1,~
\chi|_{\mathrm{Cl}(K_\mathfrak{f}/F)}\neq1\}\\
&=&\#\{\chi~\textrm{of}~\mathrm{Cl}(\mathfrak{f})~:~
\chi|_{\mathrm{Cl}(K_\mathfrak{f}/H_\mathcal{O})}=1\}
-\#\{\chi~\textrm{of}~\mathrm{Cl}(\mathfrak{f})~:~
\chi|_{\mathrm{Cl}(K_\mathfrak{f}/F)}=1\}\\
&=&\#\mathrm{Cl}(\mathfrak{f})/\mathrm{Cl}(K_\mathfrak{f}/H_\mathcal{O})-
\#\mathrm{Cl}(\mathfrak{f})/\mathrm{Cl}(K_\mathfrak{f}/F)
\quad\textrm{by Lemma \ref{quotientextend}}\\
&=&[H_\mathcal{O}:K]-[F:K]\\
&=&[H_\mathcal{O}:K](1-1/[H_\mathcal{O}:F])\\
&\geq&(1/2)[H_\mathcal{O}:K]\quad\textrm{by
the fact $F\subsetneq H_\mathcal{O}$}\\
&=&(h_K/2)\#\pi_\mathfrak{f}(\mathcal{O}_K)^*/\pi_\mathfrak{f}(\mathbb{Z})^*
\quad\textrm{from the exact sequence (\ref{exact2}) and (\ref{kernelPsi})}\\
&>&h_K\#\pi_\mathfrak{f}(\mathcal{O}_K)^*\sum_{k=1}^n\frac{1}
{\#\pi_{\mathfrak{p}_k^{e_k}}(\mathcal{O}_K)^*\cdot\#G_k}\quad\textrm{by
the assumption (\ref{assumption})}.
\end{eqnarray*}
On the other hand, we find that
\begin{eqnarray*}
&&\#\{\chi~\textrm{of}~\mathrm{Cl}(\mathfrak{f})~:~\chi|_{\mathrm{Cl}(K_\mathfrak{f}/H_\mathcal{O})}=1,~
\mathfrak{p}_k\nmid\mathfrak{f}_\chi~\textrm{for some}~k\}\\
&\leq&\#\{\chi~\textrm{of}~\mathrm{Cl}(\mathfrak{f})~:~\chi|_{\mathrm{Cl}(K_\mathfrak{f}/H_\mathcal{O})}=1,~
\widetilde{\chi}_k=1~\textrm{for some}~k\}\quad\textrm{by Lemma \ref{conductor}}\\
&=&\#\{\xi~\textrm{of}~\mathrm{Cl}(H_\mathcal{O}/K)~:~\check{\xi}_k=1~\textrm{for
some}~k\}
\quad\textrm{by the argument followed by Lemma \ref{quotientextend}}\\
&\leq&h_K\#\pi_\mathfrak{f}(\mathcal{O}_K)^*\sum_{k=1}^n\frac{1}
{\#\pi_{\mathfrak{p}_k^{e_k}}(\mathcal{O}_K)^*\cdot\#G_k}\quad\textrm{by
Lemma \ref{character2}(ii)}.
\end{eqnarray*}
Therefore, there exists a character $\chi$ of
$\mathrm{Cl}(\mathfrak{f})$ which satisfies the condition
(\ref{condition}).
\end{proof}

\begin{theorem}\label{main2}
If $\mathfrak{f}=N\mathcal{O}_K$ satisfies the assumption
\textup{(\ref{assumption})} in Lemma \textup{\ref{character3}}, then
the singular value in \textup{(\ref{ringclassinvariant})} generates
$H_\mathcal{O}$ over $K$ as a real algebraic integer.
\end{theorem}
\begin{proof}
Let
$\varepsilon=\mathbf{N}_{K_\mathfrak{f}/H_\mathcal{O}}(g_\mathfrak{f}(C_0))$
and $F=K(\varepsilon)$ as a subfield of $H_\mathcal{O}$. Suppose
that $F$ is properly contained in $H_\mathcal{O}$, then there is a
character $\chi$ of $\mathrm{Cl}(\mathfrak{f})$ satisfying the
condition (\ref{condition}) in Lemma \ref{character3}. Since
$\mathfrak{p}_k|\mathfrak{f}_k$ for all $k=1,\cdots,n$, the Euler
factor of $\chi$ in Proposition \ref{LandS} is $1$, and hence the
value $S_\mathfrak{f}(\overline{\chi},g_\mathfrak{f})$ does not
vanish by Remark \ref{remarkLandS}(ii). On the other hand, we can
derive $S_\mathfrak{f}(\overline{\chi},g_\mathfrak{f})=0$ by using
the condition (\ref{condition}) of $\chi$ in exactly the same way as
the proof of Theorem \ref{generator}, which gives rise to a
contradiction. Therefore $H_\mathcal{O}=K(\varepsilon)$, and hence
we can apply the argument of Theorem \ref{main} to complete the
proof.
\end{proof}

\begin{remark}\label{newremark}
Let $N$ ($\geq2$) be an integer with prime factorization
\begin{equation*}
N=\prod_{a=1}^A s_a^{u_a}\prod_{b=1}^B q_b^{v_b} \prod_{c=1}^C
r_c^{w_c},
\end{equation*}
where each $s_a$ (respectively, $q_b$ and $r_c$) splits
(respectively, is inert and ramified) in $K/\mathbb{Q}$. Then we
have the prime ideal factorization
\begin{equation*}
\mathfrak{f}=N\mathcal{O}_K=\prod_{a=1}^A
(\mathfrak{s}_a\overline{\mathfrak{s}}_a)^{u_a} \prod_{b=1}^B
\mathfrak{q}_b^{v_b} \prod_{c=1}^C \mathfrak{r}_c^{2w_c}
\end{equation*}
with
\begin{equation*}
\mathbf{N}_{K/\mathbb{Q}}(\mathfrak{s}_a)=
\mathbf{N}_{K/\mathbb{Q}}(\overline{\mathfrak{s}}_a)=s_a,~
\mathbf{N}_{K/\mathbb{Q}}(\mathfrak{q}_b)=q_b^2,~
\mathbf{N}_{K/\mathbb{Q}}(\mathfrak{r}_c)=r_c.
\end{equation*}
\par
Now, for the sake of convenience, we let
\begin{equation*}
\mathfrak{f}=\prod_{k=1}^{2A+B+C}\mathfrak{p}_k^{e_k}
\end{equation*}
with
\begin{equation}\label{defp}
(\mathfrak{p}_k,e_k)=\left\{\begin{array}{ll}(\mathfrak{s}_k,u_k) &
\textrm{for}~k=1,\cdots,A\\
(\overline{\mathfrak{s}}_{k-A},u_{k-A}) & \textrm{for}~k=A+1,\cdots,2A\\
(\mathfrak{q}_{k-2A},v_{k-2A}) & \textrm{for}~k=2A+1,\cdots, 2A+B\\
(\mathfrak{r}_{k-2A-B},2w_{k-2A-B}) & \textrm{for}~k=2A+B+1,\cdots,
2A+B+C,
\end{array}\right.
\end{equation}
and consider the surjection
\begin{equation*}
\mu_k=\widehat{v}_k\circ\iota^{-1}~:~\pi_\mathfrak{f}(\mathbb{Z})^*\longrightarrow
G_k~(\subseteq\pi_{\mathfrak{p}_1^{e_1}}(\mathcal{O}_K)^*\times
\cdots\times\pi_{\mathfrak{p}_{k-1}^{e_{k-1}}}(\mathcal{O}_K)^*\times
\pi_{\mathfrak{p}_{k+1}^{e_{k+1}}}(\mathcal{O}_K)^*\times
\cdots\times\pi_{\mathfrak{p}_{n}^{e_{n}}}(\mathcal{O}_K)^*).
\end{equation*}
\par
If $m\pmod{\mathfrak{f}}\in\pi_\mathfrak{f}(\mathbb{Z})^*$ belongs
to $\mathrm{Ker}(\mu_k)$, then
\begin{eqnarray}
(\underbrace{1,\cdots,1}_{n-1})&=&\mu_k(m\hspace{-0.3cm}\pmod{\mathfrak{f}})=\widehat{v}_k\circ\iota^{-1}(m\hspace{-0.3cm}\pmod{\mathfrak{f}})\nonumber\\
&=&(m\hspace{-0.3cm}\pmod{\mathfrak{p}_1^{e_1}},\cdots,m\hspace{-0.3cm}\pmod{\mathfrak{p}_{k-1}^{e_{k-1}}},
m\hspace{-0.3cm}\pmod{\mathfrak{p}_{k+1}^{e_{k+1}}},\cdots,
m\hspace{-0.3cm}\pmod{\mathfrak{p}_{n}^{e_{n}}}
),\phantom{1234}\label{kernelmu}
\end{eqnarray}
which shows
\begin{equation*}
\iota^{-1}(\mathrm{Ker}(\mu_k))\subseteq \iota_k(
\pi_{\mathfrak{p}_k^{e_k}}(\mathbb{Z})^*)=
\{(1,\cdots,1,t\hspace{-0.3cm}\pmod{\mathfrak{p}_k^{e_k}},1,\cdots,1)~:~
t\in\mathbb{Z}~\textrm{which is prime to}~\mathfrak{p}_k\}.
\end{equation*}
Hence, this gives the inequality
\begin{equation}\label{G_kbound}
\#G_k=\frac{\#\pi_\mathfrak{f}(\mathbb{Z})^*}{\#\mathrm{Ker}(\mu_k)}
\geq\frac{\#\pi_\mathfrak{f}(\mathbb{Z})^*}{\#\pi_{\mathfrak{p}_k^{e_k}}(\mathbb{Z})^*}.
\end{equation}
\par
In particular, if $k=1,\cdots,2A$, then $\mu_k$ becomes injective
(and so, bijective). Indeed, if
$m\pmod{\mathfrak{f}}\in\pi_\mathfrak{f}(\mathbb{Z})^*$ belongs to
$\mathrm{Ker}(\mu_k)$, then
\begin{equation}\label{equivrelation1}
m\equiv1\pmod{\mathfrak{p}_\ell^{e_\ell}}\quad\textrm{for}~\ell\neq
k
\end{equation}
by (\ref{kernelmu}). But, since $m$ is an integer,
(\ref{equivrelation1}) implies
\begin{equation}\label{equivrelation2}
m\equiv1\pmod{\overline{\mathfrak{p}}_\ell^{~e_\ell}}\quad\textrm{for}~\ell\neq
k.
\end{equation}
On the other hand, since
$\mathfrak{p}_k=\overline{\mathfrak{p}}_{k+A}$ or
$\overline{\mathfrak{p}}_{k-A}$ by the definition (\ref{defp}), we
deduce by (\ref{equivrelation1}) and (\ref{equivrelation2}) that
\begin{equation*}
m\equiv1\pmod{\mathfrak{p}_\ell^{e_\ell}}\quad\textrm{for
all}~\ell=1,\cdots,n,
\end{equation*}
from which we get $m\equiv1\pmod{\mathfrak{f}}$. This concludes that
$\mu_k$ is injective; hence
\begin{equation}\label{splitcase}
\#G_k=\#\pi_\mathfrak{f}(\mathbb{Z})^*\quad\textrm{for}~
k=1,\cdots,2A.
\end{equation}
\par
Thus we achieve by (\ref{G_kbound}), (\ref{splitcase}) and the Euler
functions for integers and ideals that
\begin{eqnarray*}
\textrm{(LHS) of
(\ref{assumption})}\leq4\sum_{a=1}^{A}\frac{1}{(s_a-1)s_a^{u_a-1}}
+2\sum_{b=1}^B\frac{1}{(q_b+1)q_b^{v_b-1}}
+2\sum_{c=1}^C\frac{1}{r_c^{w_c}}.
\end{eqnarray*}
Therefore, one can also apply Theorem \ref{main2} under the
assumption
\begin{equation}\label{newassumption3}
4\sum_{a=1}^{A}\frac{1}{(s_a-1)s_a^{u_a-1}}
+2\sum_{b=1}^B\frac{1}{(q_b+1)q_b^{v_b-1}}
+2\sum_{c=1}^C\frac{1}{r_c^{w_c}}<1.
\end{equation}
\end{remark}

\begin{example}\label{exer3}
Let $K=\mathbb{Q}(\sqrt{-2})$ and $\mathcal{O}$ be the order of
conductor $N=9$ ($=3^2$) in $K$. Then, $N\mathcal{O}_K$ satisfies
the assumption (\ref{newassumption3}) in Remark \ref{newremark}
(but, not the assumption (\ref{hypothesis}) in Lemma
\ref{character}), and hence the singular value
$3^{12}\Delta(9\theta)/\Delta(3\theta)$ with $\theta=\sqrt{-2}$
generates $H_\mathcal{O}$ over $K$ by Theorem \ref{main2}. Since
$h_K=1$ (\cite[p.29]{Cox}), one can estimate its minimal polynomial
in exactly the same way as Example \ref{exer1}:
\begin{eqnarray*}
&&\min(3^{12}\Delta(9\theta)/\Delta(3\theta),K)\\&=&
X^6+52079706X^5+2739284675932815X^4+12787916715651570220X^3\\
&&+190732505724302106460815X^2-268398119546256294X+1.
\end{eqnarray*}
\end{example}

\bibliographystyle{amsplain}

\begin{thebibliography}{10}

\bibitem {C-Y} I. Chen and N. Yui, \textit{Singular values of Thompson
series}, Groups, difference sets, and the Monster (Columbus, OH,
1993), 255--326, Ohio State Univ. Math. Res. Inst. Publ. 4, Walter
de Gruyter, Berlin, 1996.

\bibitem {C-K} B. Cho and J. K. Koo, \textit{Construction of class
fields over imaginary quadratic fields and applications}, Quart. J.
Math. (2010) 61 (2), 199--216.

\bibitem{Cohen} H. Cohen, \textit{Advanced Topics in Computational Number Theory},
Grad. Texts in Math. 193, Springer-Verlag, New York, 2000.

\bibitem{Cox} D. A. Cox, \textit{Primes of the form $x^2+ny^2$: Fermat, Class Field, and Complex Multiplication},
John Wiley \& Sons, Inc., New York, 1989.

\bibitem {C-M-S} D. A. Cox, J. McKay and P. Stevenhagen,
\textit{Principal moduli and class fields}, Bull. London Math. Soc.
36 (2004), no. 1, 3--12.

\bibitem {Janusz} G. J. Janusz, \textit{Algebraic Number Fields},
2nd edition, Grad. Studies in Math. 7, Amer. Math. Soc., Providence,
R. I., 1996.

\bibitem {J-K-S} H. Y. Jung, J. K. Koo and D. H. Shin, \textit{On some arithmetic properties of Siegel functions (II)},
http://arxiv.org/abs/1007.2318, submitted.

\bibitem {K-S} J. K. Koo and D. H. Shin, \textit{On some arithmetic properties of Siegel functions},
Math. Zeit. 264 (2010) 137--177.

\bibitem{K-S2} J. K. Koo and D. H. Shin, \textit{Function fields of certain arithmetic curves and application},
Acta Arith. 141 (2010), no. 4, 321--334.

\bibitem{K-L} D. Kubert and S. Lang, \textit{Modular Units}, Grundlehren der mathematischen Wissenschaften 244,
Spinger-Verlag, 1981.

\bibitem{Lang2} S. Lang, \textit{Algebraic Number Theory}, 2nd
edition, Grad. Texts in Math. 110, Springer-Verlag, New York, 1994.

\bibitem{Lang} S. Lang, \textit{Elliptic Functions}, With an appendix by J. Tate, 2nd edition, Grad. Texts in Math. 112,
Spinger-Verlag, New York, 1987.

\bibitem {Ramachandra} K. Ramachandra, \textit{Some applications of Kronecker's limit formula},
Ann. of Math. (2) 80(1964), 104--148.

\bibitem{Schertz} R. Schertz, \textit{Construction of ray class fields
by elliptic units}, J. Th\'{e}or. Nombres Bordeaux 9 (1997), no. 2,
383--394.

\bibitem{Shimura} G. Shimura, \textit{Introduction to the Arithmetic Theory of Automorphic Functions}, Iwanami Shoten and Princeton
University Press, Princeton, N. J., 1971.

\bibitem{Stevenhagen} P. Stevenhagen, \textit{Hilbert's 12th problem, complex multiplication and Shimura
reciprocity}, Class Field Theory-Its Centenary and Prospect (Tokyo,
1998), 161--176, Adv. Stud. Pure Math. 30, Math. Soc. Japan, Tokyo,
2001.

\end{thebibliography}

\end{document}